\documentclass[english]{article}
\usepackage[T1]{fontenc}
\usepackage[latin9]{inputenc}
\usepackage{geometry}
\usepackage{color}
\usepackage{float}
\usepackage{amsmath}
\usepackage{amsthm}
\usepackage{amssymb}
\usepackage{graphicx}

\makeatletter

\providecommand{\tabularnewline}{\\}
\floatstyle{ruled}
\newfloat{algorithm}{tbp}{loa}
\providecommand{\algorithmname}{Algorithm}
\floatname{algorithm}{\protect\algorithmname}

\numberwithin{equation}{section}
\theoremstyle{plain}
\newtheorem{thm}{\protect\theoremname}[section]
\theoremstyle{definition}
\newtheorem{defn}[thm]{\protect\definitionname}
\theoremstyle{definition}
\newtheorem{example}[thm]{\protect\examplename}
\theoremstyle{plain}
\newtheorem{assumption}[thm]{\protect\assumptionname}
\theoremstyle{remark}
\newtheorem{rem}[thm]{\protect\remarkname}
\theoremstyle{plain}
\newtheorem{cor}[thm]{\protect\corollaryname}
\theoremstyle{plain}
\newtheorem{lem}[thm]{\protect\lemmaname}

\makeatother

\usepackage{babel}
\providecommand{\assumptionname}{Assumption}
\providecommand{\corollaryname}{Corollary}
\providecommand{\definitionname}{Definition}
\providecommand{\examplename}{Example}
\providecommand{\lemmaname}{Lemma}
\providecommand{\remarkname}{Remark}
\providecommand{\theoremname}{Theorem}

\begin{document}
\title{A Randomized Nonlinear Rescaling Method in Large-Scale Constrained
Convex Optimization}
\author{Bo Wei, William B. Haskell, and Sixiang Zhao}
\maketitle
\begin{abstract}
We propose a new randomized algorithm for solving convex optimization
problems that have a large number of constraints (with high probability).
Existing methods like interior-point or Newton-type algorithms are
hard to apply to such problems because they have expensive
computation and storage requirements for Hessians and matrix inversions.
Our algorithm is based on nonlinear rescaling (NLR), which is a primal-dual-type
algorithm by Griva and Polyak {[{Math. Program., 106(2):237-259, 2006}]}. NLR introduces an equivalent problem through a transformation
of the constraint functions, minimizes the corresponding augmented
Lagrangian for given dual variables, and then uses this minimizer to
update the dual variables for the next iteration. The primal update
at each iteration is the solution of an unconstrained finite sum minimization
problem where the terms are weighted by the current dual variables. We use randomized
first-order algorithms to do these primal updates, for which they
are especially well suited. In particular, we use the scaled dual variables as the sampling distribution for each primal update, and we show that this distribution is the optimal one among all probability distributions. We conclude by demonstrating the favorable numerical
performance of our algorithm.
\end{abstract}

\section{Introduction}

We consider the following constrained convex optimization
problem:
\[
\mathbb{P}: \min_{x\in\mathbb{R}^{n}}\left\{ f(x): g_{i}(x)\geq0,\,\forall i\in\left\{ 1,2,\ldots,m\right\} \right\} ,
\]
where $f\text{ : }\mathbb{R}^{n}\rightarrow\mathbb{R}$ is convex,
and $g_{i}:\mathbb{R}^{n}\rightarrow\mathbb{R}$ are concave for all
$i=1,2,\ldots,m$. We are motivated by the large-scale case where $m\gg0$.

There are many algorithms for solving constrained optimization problems.
Our method is based on nonlinear rescaling (NLR) \cite{griva2006primal,polyak2001log,polyak2002nonlinear}, which is a primal-dual-type algorithm with a Q-linear convergence rate
\cite{griva2006primal,polyak2002nonlinear} (that can sometimes be
improved to a Q-superlinear convergence rate \cite{griva2006primal}). In each iteration of this method, an unconstrained augmented Lagrangian is minimized for a given set of dual variables (the ``primal update''). Then, the dual variables are updated based on the minimizer of this augmented Lagrangian (the ``dual update'').

We propose a new computational scheme where we use randomized
first-order algorithms to do the primal updates. In particular, when $m \gg 0$ is large, then the primal update amounts to solving a finite sum minimization problem with a large number
of terms. Randomized first-order algorithms are specifically designed
for this type of problems \cite{defazio2014saga,defazio2014finito,johnson2013accelerating,robbins1951stochastic,schmidt2017minimizing,xiao2014proximal}. Consequently, it is natural to combine these two ideas (NLR and randomized first-order algorithms).

\subsection{Motivation}

We motivate our study with some examples of large-scale constrained optimization problems. In radiation therapy treatment, beams of radiation are used to kill cancerous cells (known as intensity modulated radiotherapy treatment (IMRT) \cite{alber2007intensity}).
A fundamental predicament of IMRT is that it not only affects cancerous
cells, but also neighboring healthy cells. Thus, in a treatment plan,
the beams should target any tumors while limiting
radiation exposure to healthy tissue. In the case of IMRT,
the corresponding nonlinear optimization problem maximizes radiation
to the tumors under constraints that limit the exposure to healthy
cells. This problem consists of thousands of decision variables, such
as beam angles, radiation intensity, as well as tens of thousands
of constraints, which restrict the negative effects of radiation therapy
on healthy tissue.

Sequential decision-making under uncertainty, and in particular, Markov decision processes (MDP) \cite{puterman2014markov}, are generally intractable due to the curse of dimensionality. Approximate dynamic programming (ADP) provides tractable techniques for getting good policies in MDPs, often through approximate linear programming (ALP) \cite{de2003linear,de2004constraint,mohajerin2018infinite,schweitzer1985generalized}. ALP problems have a manageable number of variables, which are the weights
for a given set of basis functions. In addition, there is a constraint
for every state-action pair. Thus, the number of decision variables
is relatively small but the number of constraints is huge. ALP has
been successfully applied to inventory control
\cite{lin2019revisiting}, health care \cite{restrepo2008computational},
revenue management \cite{adelman2007dynamic}, and queuing networks
\cite{de2003linear,de2004constraint}.

Many risk-aware optimization  problems \cite{haskell2017primal,noyan2013optimization,noyan2018optimization}, and in particular, many risk-constrained optimization models \cite{dentcheva2003optimization,dentcheva2009optimization,haskell2013optimization,homem2009cutting,hu2012sample},
are essentially semi-infinite programming (SIP) problems. SIP problems
may be approximated by relaxed problems with finitely many constraints.
However, the number of constraints must be very large in order to produce
a good approximation. This phenomenon holds in general for finite relaxations of SIP problems.

\subsection{Related Works}

We summarize the literature on several methods for solving constrained optimization problems. We group these methods under the broad headings of deterministic and randomized methods.

\textit{Deterministic methods}: First-order algorithms may be used
for constrained optimization. A subgradient method for computing the
saddle-points of a convex-concave function is proposed in \cite{nedic2009subgradient},
where near-optimal primal-dual solutions are obtained. The convergence
rate of this method is $O(1/\sqrt{K})$ for both the optimality gap
and constraint violation, where $K$ is the number of iterations.
A faster primal-dual algorithm based on the drift-plus-penalty method
is proposed for general convex constrained optimization problems in
\cite{yu2016primal} which has an $O(1/K)$ convergence rate. In addition,
the dual subgradient method that averages the corresponding sequence
of primal iterates can be applied to solve general convex constrained
optimization problems with a sublinear convergence rate (see \cite{necoara2013rate,nedic2009approximate,yu2017simple}).

Barrier methods use barrier functions (see \cite{frisch1955logarithmic}
for the logarithmic barrier function and \cite{carroll1961created}
for the inverse barrier function) to find a solution to convex inequality
constrained problems by solving a sequence of unconstrained problems. As the barrier parameter goes to infinity in the unconstrained problem, the solution becomes a better approximation of the desired solution. However, the Hessian of
the barrier function tends to infinity and becomes ill-conditioned,
which makes the unconstrained minimization problem more difficult
to solve.

The augmented Lagrangian method has been proposed to solve equality constrained
problems (see \cite{hestenes1969multiplier,powell1969method}). For
inequality constrained problems, Polyak introduced the modified barrier
method \cite{polyak1992modified}. Similar to the augmented Lagrangian
method, the modified barrier method explicitly uses dual variables
to avoid ill-conditioning. The modified barrier method iteratively
minimizes a modified barrier function with respect to its primal variables,
and then it updates the dual variables. However, the modified barrier
function is not defined for all real numbers and so its implementation can lead to numerical difficulties.

Later, NLR is proposed by Polyak and Teboulle as a generalization
of the modified barrier method \cite{polyak1997nonlinear}. NLR transforms the set of constraint functions into an equivalent set by using a class of smoothing functions. NLR can solve inequality constrained problems by minimizing
the Lagrangian of the equivalent problem, and then explicitly updating
the dual variables \cite{polyak2002nonlinear}. NLR does not lead
to an unbounded increase of the scaling parameter, and thus it avoids
the ill-conditioning of the Hessian. Moreover, under the standard
second order optimality conditions, NLR converges with a Q-linear
rate for any fixed (but sufficiently large) scaling parameter \cite{polyak2001log}.

The success of the primal-dual method for linear programming (see
\cite{lustig1994computational,mehrotra1992implementation,zhang1992superlinear})
has stimulated interest in primal-dual methods for nonlinear programming
(see \cite{forsgren1998primal,vanderbei1999interior}).
The best known primal-dual method is based on the path-following paradigm
\cite{vanderbei1999interior}. Another
primal-dual method is based on NLR \cite{griva2006primal,polyak2004primal}.
This method can achieve a 1.5-Q-superlinear rate by increasing the
scaling parameter in a carefully chosen way \cite{griva2006primal}.
However, both approaches have expensive computation and storage requirements
for Hessians and matrix inversions, which make them unsuitable for
large-scale constrained optimization problems.

Very recently, inexact versions of the classical augmented Lagrangian
method are developed to solve constrained optimization problems.
In \cite{lan2016iteration}, a special class of convex optimization
problems whose feasible regions consist of a simple compact convex
set intersected with an affine manifold is considered. First-order
methods are presented based on the inexact augmented Lagrangian method,
where the subproblems are approximately solved by Nesterov's optimal
method. General convex problems with both equality and inequality
constraints are solved in \cite{xu2019iteration} also using an inexact
augmented Lagrangian method. Like \cite{lan2016iteration}, the primal
subproblems in \cite{xu2019iteration} are solved using Nesterov's optimal
method.

\textit{Randomized methods}: Randomized cutting plane algorithms have
recently been developed for constrained optimization in \cite{Calafiore_Uncertain_2005,calafiore2010random,campi2008exact,mohajerin2018infinite,esfahani2015performance}.
The idea is to input a probability distribution over
the constraints, randomly sample a modest number of constraints,
and then solve the resulting relaxed problem. Intuitively, as long
as a sufficiently large number of samples is drawn, the resulting
randomized solution should violate only a small portion of the constraints
and be nearly optimal.

In \cite{lin2019revisiting}, a convex saddle-point reformulation is proposed to solve ALP problems. A proximal stochastic mirror descent method
(PSMD) is developed which learns about regions of constraint violation via its dual update. PSMD returns a near-optimal solution and a lower bound on the cost of the optimal policy in a finite number of iterations with high probability. In \cite{wei2020inexact},
a first-order primal-dual algorithm based on Monte Carlo integration
over the constraint index set is proposed to solve general convex
SIP. Since the dual variables here are nonnegative measures on the
constraint index set, a new prox function for nonnegative measures
is needed which turns out to be a generalization of the classical
Kullback-Leibler divergence.

In \cite{hien2017inexact}, an inexact primal-dual smoothing framework
is developed for large-scale non-bilinear saddle point problems, in
which randomized algorithms are used to solve the primal and
dual subproblems. As an important application, this framework is applied
to solve convex optimization problems with many constraints. In \cite{xu2018primal},
a primal-dual stochastic gradient method is developed for problems with a stochastic objective and many functional constraints.

\subsection{Main Contributions}

We highlight the three main contributions of our present work as follows:
\begin{enumerate}
\item We do the primal updates for NLR in a new way using randomized first-order
algorithms for unconstrained minimization of the augmented Lagrangian. We call our new algorithm 'Randomized Nonlinear Rescaling' (RanNLR).  RanNLR supports any randomized first-order algorithm as its subroutine, e.g. SGD \cite{robbins1951stochastic},
and variance reduction methods (e.g. SVRG \cite{johnson2013accelerating,xiao2014proximal},
SAGA \cite{defazio2014saga}, SAG \cite{schmidt2017minimizing}, Finito
\cite{defazio2014finito}).
\item We do adaptive random constraint sampling for the primal updates by
constructing a probability distribution over the constraints based
on the current dual variables (i.e., this distribution changes as
the dual variables are updated). We show that random sampling from this distribution is optimal compared to any other possible sampling distribution.
\item We analyze the complexity of RanNLR required to obtain a solution
within distance $\varepsilon$ of the optimal solution of Problem $\mathbb{P}$
when the objective $f$ is strongly convex, with probability at least
$1-\delta$. If the primal update subroutine has a sublinear rate (e.g. SGD),
then the overall complexity of our algorithm is $\tilde{O}\left(1/(\varepsilon^{2}\delta)\right)$
(where $\tilde{O}(\cdot)$ hides the $\ln(1/\varepsilon)$ and $\ln(1/\delta)$
factors). If the primal update subroutine has a linear rate (e.g. SVRG), then
the overall complexity of our algorithm is $O\left(\ln(1/\varepsilon)(2\ln(1/\varepsilon)+\ln(1/\delta))\right)$.
\end{enumerate}
This paper is organized as follows. In Section \ref{sec:Preliminaries},
we review classical NLR. We then present the details of our new randomized NLR algorithm in Section \ref{sec:A-randomized Nonlinear-Rescaling}.
Our main results including the complexity analysis of RanNLR may be found
in Section \ref{sec:Main-result}. Then, we present numerical experiments in Section \ref{sec:Numerical-Experiments}, and conclude the paper in Section
\ref{sec:Conclusion}. Supporting technical results are gathered together in the Appendix.

\textbf{Notation. }Let $\mathbb{N}$ be the set of natural numbers.
For a positive integer $n\in\mathbb{N}$, let $[n]\triangleq\{1,2,\ldots,n\}$
and $[n]_{0}\triangleq\{0,1,\ldots,n-1\}$. For a real number $x$,
let $\left\lceil x\right\rceil $ be the ceiling of $x$, i.e., the
smallest integer greater than or equal to $x$.

Let $\mathbb{R}^{n}$ be $n$-dimensional Euclidean space. Let $\mathbb{R}_{+}^{n}$ and $\mathbb{R}_{++}^{n}$ be the subsets of vectors in $\mathbb{R}^{n}$
with nonnegative and strictly positive components, respectively. For a vector $x=(x_{1},x_{2},\ldots,x_{n})\in\mathbb{R}^{n}$,
define $\|x\|_{1}\triangleq\sum_{i=1}^{n}\left|x_{i}\right|$, $\|x\|_{2}\triangleq\sqrt{\sum_{i=1}^{n}x_{i}^{2}}$,
and $\left\Vert x\right\Vert _{\infty}\triangleq\max_{1\leq i\leq n}\left|x_{i}\right|$.
For vectors $x,\,y\in\mathbb{R}^{n}$, $\left\langle x,y\right\rangle$
denotes the Euclidean inner product.

Let $\mathbb{R}^{n_{1}\times n_{2}}$ be the set of matrices with
dimension $n_{1}\times n_{2}$. Let $\mathrm{I}_{n}\in\mathbb{R}^{n\times n}$ be the identity matrix. For a matrix $A=(a_{ij})_{i\in[n_{1}],j\in[n_{2}]}\in\mathbb{R}^{n_{1}\times n_{2}}$,
let $A^{T}\in\mathbb{R}^{n_{2}\times n_{1}}$ be its transpose, and
define the matrix norm $\left\Vert A\right\Vert \triangleq\left\Vert A\right\Vert _{\infty}\triangleq\max_{1\leq i\leq n_{1}}\sum_{j=1}^{n_{2}}\left|a_{ij}\right|$
which is the maximum absolute row sum. For a vector $a=(a_{1},\ldots,a_{n})\in\mathbb{R}^{n}$,
define
\[
\mathrm{diag}(a)\triangleq\left(\begin{array}{ccc}
a_{1} & 0 & 0\\
0 & \ddots & 0\\
0 & 0 & a_{n}
\end{array}\right)\in\mathbb{R}^{n\times n}.
\]
For $\eta>0$ and $x\in\mathbb{R}^{n}$, let $B_{\eta}(x)\triangleq\left\{ z\in\mathbb{R}^{n}:\left\Vert z-x\right\Vert _{\infty}\leq\eta\right\} $
denote the Euclidean ball in $\mathbb{R}^{n}$ with radius $\eta$
in the $\left\Vert \cdot\right\Vert _{\infty}$-norm centered at $x$. For a set $C\subset\mathbb{R}^{n}$, let $I_{C}$ denote the indicator
function of the set $C$. The projection operator $\Pi_{\mathcal{X}}:\mathbb{R}^{n}\rightarrow\mathbb{R}^{n}$
(for a closed convex set $\mathcal{X}\subset\mathbb{R}^{n}$) is given
by $\Pi_{\mathcal{X}}[x]\triangleq\arg\min_{y\in\mathcal{X}}\|x-y\|_{2}$, which always exists and is unique.

\section{\label{sec:Preliminaries} Nonlinear Rescaling (NLR)}

We begin by reviewing the details of classical NLR, on which our present method is based. Classical NLR employs a nonlinear rescaling function scaled by $N>0$ (hereafter
called the "scaling parameter") to transform each constraint function.
We keep $N$ constant throughout the course of our algorithm, and
we will see that the overall convergence rate depends on $N$ \cite{griva2006primal,polyak2001log,polyak2002nonlinear}.
We detail the effect of $N$ on the algorithm complexity and offer
selection guidelines later, for now we just treat $N$ as a
constant.

Rescaling each constraint function gives a new problem that is equivalent to Problem $\mathbb{P}$. Furthermore, the
Lagrangian for this new problem can be viewed as an augmented Lagrangian
for Problem $\mathbb{P}$. In each iteration of NLR, we minimize this
augmented Lagrangian for given dual variables, obtain a solution,
and then use this solution to update the dual variables. The specific
properties of the nonlinear rescaling function substantially affect
both the global and local behavior of the overall algorithm (this
phenomenon is characterized in Lemma~\ref{lem:one step NR}).
\begin{defn}
\label{def:NR function}\cite{griva2006primal} A {\em nonlinear
rescaling function} $\psi:\mathbb{R}\rightarrow\mathbb{R}$ is a
twice continuously differentiable function such that: (i) $\psi(0)=0$;
(ii) $\psi'(t)>0$ for all $t\in\mathbb{R}$ and $\psi'(0)=1$; (iii)
$\psi''(t)<0$ for all $t\in\mathbb{R}$; (iv) $\psi(t)\leq-at^{2}$
for some $a>0$ and all $t\leq0$; (v) $\psi'(t)\leq d_{1}t^{-1}$
and $-\psi''(t)\leq d_{2}t^{-2}$ for some $d_{1}>0,d_{2}>0$, and
all $t>0$. Let $\Psi$ denote the class of all nonlinear rescaling
functions.
\end{defn}

Some examples of nonlinear rescaling functions in $\Psi$ follow.
\begin{example}
Define $\zeta_{1}(t)\triangleq1-e^{-t}$ for $t\in\mathbb{R}$, $\zeta_{2}(t)=\ln(t+1)$
for $t>-1$, and $\zeta_{3}(t)=t/(t+1)$ for $t\in\mathbb{R}\setminus\{-1\}$.
For $\tau\in(-1,0)$, the quadratic extrapolation of $\zeta_{i}$
(see \cite{griva2006primal}) is defined as:
\[
\psi_{i}(t)\triangleq\begin{cases}
\zeta_{i}(t), & t\geq\tau,\\
0.5\zeta_{i}''(\tau)t^{2}+(\zeta_{i}'(\tau)-\tau\zeta_{i}''(\tau))t+\zeta_{i}(\tau)-\tau\zeta_{i}'(\tau)+\tau^{2}\zeta_{i}''(\tau), & t\leq\tau,
\end{cases}
\]
for $i=1,2,3$. We can directly verify that $\psi_{i}\in\Psi$ for
all $i=1,2,3$. 
%In Figure \ref{fig:phi}, we plot
%\[
%\psi_{1}(t)=\begin{cases}
%1-e^{-t}, & t\geq-0.5,\\
%-0.5e^{0.5}t^{2}+0.5e^{0.5}t+1-\frac{5}{8}e^{0.5}, & %t\leq-0.5.
%\end{cases}
%\]

%\begin{figure}
%\centering \includegraphics[width=0.6\textwidth,height=0.%2\paperheight]{figofphi}\caption{\label{fig:phi}$\psi_{1}%(t)$ with $\tau=-0.5$.}
%\end{figure}
\end{example}

For any $\psi\in\Psi$ and scaling parameter $N>0$, we define the smoothed optimization problem:
\begin{align*}
\mathbb{P}_{N}:\min_{x\in\mathbb{R}^{n}}\left\{f(x): N^{-1}\psi(Ng_{i}(x))\geq0,\forall i\in[m]\right\}.
\end{align*}
Problem $\mathbb{P}_{N}$ is a convex optimization problem due to properties (i), (ii), and (iii) of Definition \ref{def:NR function}. Additionally,
Problem $\mathbb{P}_{N}$ is equivalent to Problem $\mathbb{P}$ in
the sense that both share the same feasible region, optimal solutions,
and optimal value.

For Problem $\mathbb{P}_{N}$, let $\lambda\in\mathbb{R}_{++}^{m}$ be the dual variables corresponding to all $m$ inequality constraints, and let $\mathcal{L}_{N}(x,\lambda)\triangleq f(x)-N^{-1}\sum_{i\in[m]}\lambda_{i}\psi(Ng_{i}(x))$ be the Lagrangian (which is also an augmented Lagrangian for Problem $\mathbb P$). For fixed $\lambda\in\mathbb{R}_{++}^{m}$, the ``primal update'' is to minimize $x\mapsto\mathcal{L}_{N}(x,\lambda)$ which we denote as:
\[
\mathbb{P}_{N}(\lambda): \min_{x\in\mathbb{R}^{n}}\mathcal{L}_{N}(x,\lambda).
\]

Classical NLR proceeds as follows. We let $k\geq0$ count iterations, $\lambda^{k}=(\lambda_{i}^{k})_{i\in\left[m\right]}\in\mathbb{R}_{++}^{m}$ denote the dual variables in iteration $k$, and $x^k$ denote the primal variables in iteration $k$.
In iteration $k$, starting with $\lambda^{k}$, we compute $x_{*}^{k+1}(\lambda^{k})$ by solving Problem $\mathbb{P}_{N}(\lambda^{k})$ exactly:
\begin{align}
x_{*}^{k+1}(\lambda^{k}) \in \arg\min_{x\in\mathbb{R}^{n}}\mathcal{L}_{N}(x,\lambda^{k}).\label{eq:exact primal solution}
\end{align}
Eq.~\eqref{eq:exact primal solution} is the primal update of NLR, see \cite[Eqs.\ (3.4)-(3.7)]{griva2006primal}. Next, we do the dual update:
\begin{align}
\lambda_{i}^{k+1}=\lambda_{i}^{k}\psi'(Ng_{i}(x_{*}^{k+1}(\lambda^{k}))), & \quad\forall\, i\in[m].\label{eq:exact dual update}
\end{align}
Eq.~\eqref{eq:exact dual update} can be expressed more compactly in vector notation as $\lambda^{k+1}=\lambda^{k}\psi'(N\,G(x^{k+1}))$.

\section{\label{sec:A-randomized Nonlinear-Rescaling} Randomized Nonlinear
Rescaling (RanNLR)}

The primal update in Eq.~(\ref{eq:exact primal solution}) is usually the bottleneck in NLR, especially when $m\gg0$. RanNLR builds on NLR by allowing the primal update to be done inexactly using a randomized first-order algorithm. The pseudo-code of RanNLR is summarized in Algorithm \ref{alg: randomized nonlinear rescaling}. We continue to let $k$ index the outer iterations of RanNLR, the same as for NLR.

We define, for all $\lambda\in\mathbb{R}_{++}^{m}$, the terms:
\begin{align*}
f_{i}^{N}(x;\lambda)\triangleq f(x)-\|\lambda\|_{1}N^{-1}\psi(Ng_{i}(x)), & \quad\forall i\in[m].
\end{align*}
Then, we may write Problem $\mathbb{P}_{N}(\lambda)$ as an explicit
finite sum minimization problem:
\begin{align}
\mathbb{P}_{N}(\lambda)\equiv\min_{x\in\mathbb{R}^{n}}\left\{ \mathcal{L}_{N}(x,\lambda)\equiv\sum_{i\in[m]}\frac{\lambda_{i}}{\|\lambda\|_{1}}f_{i}^{N}(x;\lambda)\right\} .\label{eq:finite_sum}
\end{align}
The purpose of this reformulation is twofold. First, it absorbs the
original objective function $f$ into the $m$ functions $\{f_{i}^{N}\}_{i\in[m]}$.
Second, it shows how the dual variables determine an explicit probability distribution
over the constraint index set $[m]$. Whenever we solve an instance of Problem $\mathbb{P}_{N}(\lambda)$,
we solve Problem~(\ref{eq:finite_sum}) specifically.

\begin{algorithm}
\caption{\label{alg: randomized nonlinear rescaling} Randomized Nonlinear
Rescaling ($N$, $x^{0}$, $\lambda^{0}$, $K$, $\epsilon$, $\delta$)}

\textbf{\textcolor{black}{Input:}} Total number of iterations $K\geq1$,
scaling parameter $N>0$, error tolerance $\epsilon>0$, overall failure
probability $\delta\in(0,1)$.

\textbf{\textcolor{black}{Initialize: }}$x^{0}\in\mathbb{R}^{n}$,
$\lambda^{0}\in\mathbb{R}_{++}^{m}$.

$\mathbf{For}$ $k=0,1,\ldots,K-1$:
\begin{itemize}
\item Use subroutine $\mathcal{A}$ to compute $x^{k+1}\in\mathbb{R}^{n}$ such that
$\left\Vert \nabla_{x}\mathcal{L}_{N}(x^{k+1},\lambda^{k})\right\Vert _{\infty}\leq\epsilon$
with probability at least $(1-\delta)^{1/K}$;
\item Update $\lambda^{k+1}$ by Eq.~(\ref{eq:dual update}).
\end{itemize}
$\mathbf{End}$

\textbf{\textcolor{black}{Return:}} $x^{K}$.
\end{algorithm}

\subsection{The Subroutine $\mathcal{A}$ \label{subsec:subroutine A}}

We now let $\mathcal{A}$ denote a general randomized first-order subroutine for solving Problem (\ref{eq:finite_sum}). Some specific examples of $\mathcal{A}$ include: SGD \cite{robbins1951stochastic},
SVRG \cite{johnson2013accelerating,xiao2014proximal}, SAGA \cite{defazio2014saga},
SAG \cite{schmidt2017minimizing}, and Finito \cite{defazio2014finito}.
These latter four variance reduction algorithms combine the advantages
of full gradient descent and SGD to achieve a linear convergence rate
in expectation while maintaining the low per-iteration cost of SGD.

The dual variables $\lambda\in\mathbb{R}_{++}^{m}$ are fixed in each instance of Problem $\mathbb{P}_{N}(\lambda)$. They enter into the sampling distribution of $\mathcal{A}$
and weights of the finite sum minimization Problem~(\ref{eq:finite_sum}). Define $\mathcal P_{+}([m])$ to be the set of all probability distributions on $[m]$ with all positive components. The subroutine $\mathcal A$ relies on a sampling distribution from $\mathcal P_{+}([m])$. For easy reference, we denote this sampling distribution as $\wp\left(\lambda\right)\triangleq\left(\lambda_{i}/\|\lambda\|_{1}\right)_{i\in[m]}\in \mathcal P_{+}([m])$, and its components as $\wp_{i}\left(\lambda\right)\triangleq\lambda_{i}/\|\lambda\|_{1}$ for all $i \in [m]$.

For each outer iteration $k$, $\lambda^k$ is fixed and we want to solve the corresponding Problem $\mathbb{P}_{N}(\lambda^k)$. We let $t\geq0$ index the inner iterations of $\mathcal A$ applied to solve $\mathbb{P}_{N}(\lambda^k)$. We also let $\{I_{t}\}_{t\geq0}$ be a sequence of i.i.d.\ random variables drawn from $[m]$ according to $\wp(\lambda^k)$. This sampling distribution is adaptive, it changes as $\lambda^k$ varies. We further explain our adaptive sampling scheme and its advantages in Subsection~\ref{sec:Advantage-of-adaptive}.

Our analysis requires $\wp(\lambda^k) \in \mathcal P_{+}([m])$ to hold for all $k \geq 0$. In fact, in our scheme, when we initialize with $\lambda^0 \in \mathbb{R}_{++}^{m}$, then every subsequent dual iterate will remain in $\mathbb{R}_{++}^{m}$.

In our implementation, we take $\mathcal{A}$ to be SGD and SVRG. SGD has a sublinear convergence rate (in expectation), but its complexity does not depend on $m$. SVRG offers a linear convergence rate (in expectation), but its complexity does depend on $m$ since it does a full gradient update at the beginning of each epoch.

\subsection{\label{subsec:nonlinear rescaling method} Relaxed Stopping Condition}

In classic NLR, the primal update requires each instance of Problem
$\mathbb{P}_{N}(\lambda^{k})$ to be solved \textit{exactly}. Since we are using a randomized subroutine, we will
instead solve each instance of Problem $\mathbb{P}_{N}(\lambda^{k})$
inexactly (with high probability). Specifically, we use the following relaxed
stopping condition for the primal update:
\begin{align}
\text{Find }x^{k+1}\in\mathbb{R}^{n}\text{ s.t. }\Vert\nabla_{x}\mathcal{L}_{N}(x,\lambda^{k})\Vert_{\infty}\leq\epsilon,\label{eq:inexact primal update}
\end{align}
for some small $\epsilon>0$. Any $x^{k+1}$ satisfying Eq.~\eqref{eq:inexact primal update}
is a \textit{near-optimal} solution of Problem $\mathbb{P}_{N}(\lambda^{k})$. Once we have an $x^{k+1}$ satisfying Eq.~(\ref{eq:inexact primal update}),
we update the dual variables in the usual way via:
\begin{align}
\lambda^{k+1}=\lambda^{k}\psi'(N\,G(x^{k+1})),\label{eq:dual update}
\end{align}
which is still deterministic.

\begin{rem}
Eq.~\eqref{eq:inexact primal update} is different from the inexact stopping
conditions in \cite[Eq.~(7.1)]{polyak2001log} and \cite[Eq.~(3.9)]{griva2006primal},
which require $x^{k+1}$ to satisfy:
\[
\left\Vert \nabla_{x}\mathcal{L}_{N}(x^{k+1},\lambda^{k})\right\Vert _{\infty}\leq a\,N^{-1}\left\Vert \lambda^{k}\psi'(N\,G(x^{k+1}))-\lambda^{k}\right\Vert _{\infty},
\]
for large enough $N>0$ and some $a>0$. Eq.~\eqref{eq:inexact primal update}
is more convenient for us because we can explicitly determine the
number of iterations of $\mathcal{A}$ required to meet this condition.
\end{rem}

\section{\label{sec:Main-result} Main Results}

We give the convergence analysis for RanNLR in this section. First, we gather all of the technical assumptions on $\mathbb P$ and $\mathcal A$ in the following two subsections for easy reference. Then, we give our main results and proof, followed by a justification of our adaptive sampling scheme.

\subsection{\label{subsec:Technical-assumptions} Assumptions on Optimization Problem}

We begin with basic assumptions on Problem $\mathbb{P}$ itself.
\begin{assumption}
\label{solandSlater} (i) (Solvability) There exists an optimal solution
$x^{*}$ of Problem $\mathbb{P}$.\\
 (ii) (Slater condition) There exists a Slater point $\tilde{x}\in\mathbb{R}^{n}$
such that $\kappa\triangleq\min_{i\in[m]} g_{i}\left(\tilde{x}\right)>0$ (i.e. $g_i(\tilde x) \geq \kappa > 0$ for all $i \in [m]$).
\end{assumption}
Next, we make the following assumptions on the ingredients of Problem
$\mathbb{P}$.
\begin{assumption}
\label{assu:functions assump} (i) The objective function $f\text{ : }\mathbb{R}^{n}\rightarrow\mathbb{R}$
is strongly convex with parameter $\mu_{f}>0$ with respect to $\|\cdot\|_{2}$,
i.e., for any $x_{1},x_{2}\in\mathbb{R}^{n}$, $f(x_{1})\geq f(x_{2})+\langle\nabla f(x_{2}),x_{1}-x_{2}\rangle+\mu_{f}\|x_{1}-x_{2}\|_{2}^{2}/2$. The objective function $f$ is Lipschitz continuous with parameter
$L_{f}^{(0)}\geq0$ with respect to $\|\cdot\|_{\infty}$, i.e.,  for any $x_{1},x_{2}\in\mathbb{R}^{n}$,
$|f(x_{1})-  f(x_{2})| \leq L_{f}^{(0)}\left\Vert x_{1}-x_{2}\right\Vert _{\infty}$. The Hessian $\nabla^{2}f(\cdot)$ is Lipschitz continuous with parameter
$L_{f}^{(2)}\geq0$ with respect to $\|\cdot\|_{\infty}$, i.e., for
any $x_{1},x_{2}\in\mathbb{R}^{n}$, $\left\Vert \nabla^{2}f(x_{1})-\nabla^{2}f(x_{2})\right\Vert \leq L_{f}^{(2)}\left\Vert x_{1}-x_{2}\right\Vert _{\infty}$.

(ii) For all $i\in[m]$, $g_{i}:\mathbb{R}^{n}\rightarrow\mathbb{R}$
is concave and is Lipschitz continuous with parameter $L_{g}^{(0)}\geq0$
with respect to $\|\cdot\|_{\infty}$ uniformly in $i\in[m]$, i.e., for any $x_{1},x_{2}\in\mathbb{R}^{n}$, $|g_{i}(x_{1})-  g_{i}(x_{2})| \leq L_{g}^{(0)}\left\Vert x_{1}-x_{2}\right\Vert _{\infty}$ for all $i \in [m]$. For all $i\in[m]$, the Hessian $\nabla^{2}g_{i}(\cdot)$ is Lipschitz
continuous with parameter \textup{$L_{g}^{(2)}\geq0$} with respect
to $\|\cdot\|_{\infty}$ uniformly in $i\in[m]$, i.e., for any $x_{1},x_{2}\in\mathbb{R}^{n}$,
$\left\Vert \nabla^{2}g_{i}(x_{1})-\nabla^{2}g_{i}(x_{2})\right\Vert \leq L_{g}^{(2)}\left\Vert x_{1}-x_{2}\right\Vert _{\infty}$ for all $i \in [m]$.
\end{assumption}

\begin{rem}
The assumption on Lipschitz continuity of the Hessians $\nabla^{2}f(\cdot)$
and $\nabla^{2}g_{i}(\cdot)$ for all $i\in[m]$ also appears in
the literature (see \cite[Eq. (7.5)]{polyak2001log} and \cite[Eq. (3.11)]{griva2006primal}).
\end{rem}

Due to the strong convexity of $f$, the optimal solution of Problem
$\mathbb{P}$ (and Problem $\mathbb{P}_{N}$) is unique. Similarly,
Problem $\mathbb{P}_{N}(\lambda)$ always has a unique solution for
any $\lambda\in\mathbb{R}_{++}^{m}$ because its objective function
$\mathcal{L}_{N}(x,\lambda)$ is strongly convex in $x$ as well (also
due to strong convexity of $f$).

The original Lagrangian for Problem $\mathbb{P}$ is defined by $L(x,\lambda)\triangleq f(x)-\sum_{i\in[m]}\lambda_{i}g_{i}(x)$.
The KKT conditions for an optimal solution $x^{*}$ imply that there
exists a nonnegative vector $\lambda^{*}=(\lambda_{1}^{*},\ldots,\lambda_{m}^{*})$
such that $\nabla_{x}L(x^{*},\lambda^{*})=0$, and $\lambda_{i}^{*}g_{i}\left(x^{*}\right)=0$
for all $i\in[m]$. Define the dual function $d(\lambda)\triangleq\inf_{x\in\mathbb{R}^{n}}L(x,\lambda)$.
Under Assumption \ref{solandSlater}(ii), for any $\tilde{\lambda}\in\mathbb{R}_{+}^{m}$ we have $\left\Vert \lambda^{*}\right\Vert _{\infty}\leq(f(\tilde{x})-d(\tilde{\lambda}))/\kappa$
(see \cite[Lemma 1]{nedic2009approximate}).

Under Assumption \ref{solandSlater}(ii), the KKT conditions have a solution. To characterize the KKT conditions, let $I^{*}\triangleq\left\{ i\in[m]:g_{i}(x^{*})=0\right\}$
denote the set of active constraints at $x^{*}$ (this set $I^{*}$
is unique since $x^{*}$ is unique). For brevity in notation, we define the vector-valued functions $G:\mathbb{R}^{n}\rightarrow\mathbb{R}^{m}$
via $G(x)\triangleq\left(g_{i}(x)\right)_{i\in[m]}$, $G_{I^{*}}:\mathbb{R}^{n}\rightarrow\mathbb{R}^{\left|I^{*}\right|}$
via $G_{I^{*}}(x)\triangleq\left(g_{i}(x)\right)_{i\in I^{*}}$, and
the Jacobian $\nabla G_{I^{*}}:\mathbb{R}^{n}\rightarrow\mathbb{R}^{\left|I^{*}\right|\times n}$
via $\nabla G_{I^{*}}(x)\triangleq\left(\nabla g_{i}(x)\right)_{i\in I^{*}}^{T}$.
We assume the following regularity condition holds at the optimal
solution $x^{*}$.
\begin{assumption}
\label{regularity} We have $\mathrm{rank}\left(\nabla G_{I^{*}}(x^{*})\right)=\left|I^{*}\right|$
and $\lambda_{i}^{*}>0$ for all $i\in I^{*}$, where $\left|I^{*}\right|$
is the cardinality of $I^{*}$.
\end{assumption}

\begin{rem}
In addition to Assumption \ref{regularity}, the standard second-order
optimality sufficient conditions (see \cite{griva2006primal,polyak2001log,polyak2002nonlinear})
require that, for all $y\neq0$ satisfying $\nabla G_{I^{*}}(x^{*})y=0$
(i.e., those vectors in the nullspace of $\nabla G_{I^{*}}(x^{*})$),
there exists a constant $\rho>0$ such that
\begin{equation}
\left\langle \nabla_{xx}^{2}L(x^{*},\lambda^{*})y,y\right\rangle \geq\rho\left\langle y,y\right\rangle .\label{eq:isolate}
\end{equation}
The condition in (\ref{eq:isolate}) holds with $\rho=\mu_{f}$ by
strong convexity.
\end{rem}

Finally, we split the dual optimal vector $\lambda^{*}$ into the
active $\lambda_{I^{*}}^{*}=\left(\lambda_{i}^{*}\right)_{i\in I^{*}}\in\mathbb{R}_{++}^{\left|I^{*}\right|}$
and inactive $\lambda_{[m]\setminus I^{*}}^{*}=\left(\lambda_{i}^{*}\right)_{i\in[m]\setminus I^{*}}=0^{[m]\setminus\left|I^{*}\right|}$
parts. Define $\sigma\triangleq\min\left\{ g_{i}(x^{*}):i\in[m]\setminus I^{*}\right\} >0$,
$\varLambda_{I^{*}}^{*}\triangleq\mathrm{diag}(\lambda_{I^{*}}^{*})$,
and
\[
\Phi_{N}(x^{*},\lambda^{*})\triangleq\left[\begin{array}{cc}
\nabla_{xx}L(x^{*},\lambda^{*}), & -\nabla G_{I^{*}}(x^{*})^{T}\\
-\left\langle \varLambda_{I^{*}}^{*},\nabla G_{I^{*}}(x^{*})\right\rangle , & \psi''(0)^{-1}N^{-1}\mathrm{I}_{\left|I^{*}\right|}
\end{array}\right].
\]
From \cite[p.~186]{polyak1992modified} or \cite[p.~442]{polyak2001log},
we note that the inverse matrix $\Phi_{N}(x^{*},\lambda^{*})^{-1}$
exists, and there is a (large enough) number $N_{0}>0$ and a number
$C_{\Phi}>0$ such that
\begin{equation}
\|\Phi_{N}(x^{*},\lambda^{*})^{-1}\|\leq C_{\Phi},\quad\forall\, N\geq N_{0}.\label{eq:inverseofPHIbounded}
\end{equation}
From property (v) of the definition of $\Psi$, there is a constant
$C_{\nabla}>0$ such that
\[
\sum_{i\in[m]\setminus I^{*}}4\psi'(N\sigma/2)\left\Vert \nabla g_{i}(x^{*})\right\Vert _{\infty}\leq C_{\nabla}N^{-1}.
\]
We also define a constant $c_{R}\triangleq\max\left\{ 2d_{1}\sigma^{-1},2(C_{\nabla}-2\psi''(0)^{-1})C_{\Phi}\right\} $
for later use.

We pause to note that $\sigma$ and $C_{\Phi}$ are unknown constants which are intrinsic to the NLR method. These unknown constants also appear generally in the NLR literature \cite{griva2006primal,polyak2001log}. In acknowledgement of these unknown constants, we emphasize that our upcoming main result gives the theoretical order of convergence of RanNLR.

\subsection{Assumptions on the Subroutine $\mathcal A$}

We need the following assumption for the convergence analysis of the
subroutine $\mathcal{A}$.
\begin{assumption}
\label{assu:conceptual} (i) There exists a compact set $\mathcal{X}\subseteq\mathbb{R}^{n}$
such that all primal iterates of inexact NLR lie within $\mathcal{X}$.\\
(ii) For all $\lambda\in\mathbb{R}_{++}^{m}$ and $i\in[m]$, the gradient
$\nabla_{x}f_{i}^{N}(\cdot;\lambda)$ is Lipschitz continuous with
parameter $L_{N}(\lambda)$ with respect to $\|\cdot\|_{2}$ over
$\mathcal{X}$.
\end{assumption}

The existence of such a compact set $\mathcal{X}\subseteq\mathbb{R}^{n}$ will be confirmed by Lemma \ref{lem:bddness under inexact NR}. Specifically, when the
feasible region $\left\{x\in\mathbb{R}^{n}:g_{i}(x)\geq0,\,\forall i\in[m]\right\}$ lies within a centered ball with radius $\varsigma$, as long as the scaling parameter $N>0$ is sufficiently large and the
error tolerance $\epsilon>0$ in our inexact stopping condition Eq.~(\ref{eq:inexact primal update}) is small enough, then all primal iterates of inexact NLR will lie within the centered ball with radius $2\varsigma$.

From the smoothness of objective and constraint functions in Assumption
\ref{assu:functions assump} and the continuity of $\psi'$ and $\psi''$,
we directly have
\[
\nabla_{xx}f_{i}^{N}(x,\lambda)=\nabla^{2}f(x)-\|\lambda\|_{1}\left(\psi''(Ng_{i}(x))N\nabla g_{i}(x)\nabla^{T}g_{i}(x)+\psi'(Ng_{i}(x))\nabla^{2}g_{i}(x)\right),\;\forall i\in[m].
\]
Since $\mathcal{X}$ is a compact set, we may take $L_{N}(\lambda) = \sup_{x\in\mathcal{X}}\left\Vert \nabla_{xx}f_{i}^{N}(x,\lambda)\right\Vert _{2}$.

Based on Assumption~\ref{assu:conceptual}(i), we restrict the iterates to $\mathcal{X}$ when we implement $\mathcal{A}$ to solve Problem~(\ref{eq:finite_sum}). In each iteration, the goal is for $\mathcal{A}$ to return a nearly optimal solution to $\mathbb{P}_{N}(\lambda)$ (based on our relaxed stopping condition) with a high probability. By the chain rule for the probability of the intersection of events,
the primal updates in all iterations of RanNLR will then be nearly optimal with a high probability.
\begin{rem}
If we include nonnegativity constraints $x_{i}\geq0$ for all
$i\in[n]$, then the augmented Lagrangian $\mathcal{L}(x,\lambda)$
is strongly convex on any bounded set in $\mathbb{R}^{n}$ (see \cite[Lemma 2]{polyak2002nonlinear}).
In this case, we would not need strong convexity of $f$ for our analysis to go through.
\end{rem}

In the next assumption, we formalize the two possible cases for the convergence rate of $\mathcal{A}$ for the distance to the optimal solution: the sublinear case (e.g. SGD) and the linear case (e.g. SVRG).
\begin{assumption}
\label{assu:rate} Suppose $F:\mathbb{R}^{n}\rightarrow\mathbb{R}$
is $\mu-$strongly convex with respect to $\|\cdot\|_{2}$,
and the gradient $\nabla F(\cdot)$ is $L-$Lipschitz continuous with respect to $\|\cdot\|_{2}$ over $\mathcal{X}$. Let $\{y_{t}\}_{t\geq0}$ be
the iterates of $\mathcal{A}$ applied to compute
$y^{*}\in\arg\min_{x\in\mathcal{X}}F(x)$,
where $y^{*}$ is the unique minimum of $F$.

(i) (Sublinear rate) $\mathbb{E}\left[\left\Vert y_{t}-y^{*}\right\Vert _{2}^{2}\right]\leq O(1/t)$ for all $t\geq1$. In particular, there exist $A=A(\mu,\,L)>0$ and $B=B(\mu,\,L)>0$
such that $\mathbb{E}\left[\left\Vert y_{t}-y^{*}\right\Vert _{2}^{2}\right]\leq \left(A\left\Vert y_{0}-y^{*}\right\Vert _{2}^{2}+B\right)/t$ for all $t\geq1$.

(ii) (Linear rate) For some $\alpha\in(0,1)$ and all $t\geq1$, $\mathbb{E}\left[\left\Vert y_{t}-y^{*}\right\Vert _{2}^{2}\right]\leq O(\alpha^{t})$. In particular, there exist $\zeta=\zeta(\mu,\,L)>0$ and $\alpha=\alpha(\mu,\,L)\in(0,1)$
such that 
$\mathbb{E}\left[\left\Vert y_{t}-y^{*}\right\Vert _{2}^{2}\right]\leq\zeta\,\alpha^{t}\left\Vert y_{0}-y^{*}\right\Vert _{2}^{2}$ for all $t\geq1$.
\end{assumption}

To solve Problem $\mathbb{P}_{N}(\lambda)$, we implement $\mathcal{A}$
by sampling from $\wp(\lambda)\in \mathcal P_{+}([m])$ (the probability distribution
scaled from the current dual variable). Let $\mathbb{E}^{\wp(\lambda)}\left[\cdot\right]$
denote conditional expectation with respect to $\wp(\lambda)$. We denote $\mathcal{A}_{SGD}$ and $\mathcal{A}_{SVRG}$ as SGD and SVRG with sampling distribution $\wp(\lambda)$, respectively. 
\begin{example}
(SGD, see \cite{robbins1951stochastic}) The iterates follow $y_{t+1}=\Pi_{\mathcal{X}}[y_{t}-\gamma_{t}\nabla f_{I_{t}}^{N}(y_{t},\lambda)]$,
for all $t\geq0$. Let $M_{B}\triangleq\max_{i\in[m]}\left\Vert \nabla f_{i}^{N}(x^{*}(\lambda),\lambda)\right\Vert _{2}$.
By Remark \ref{rem:special cases of sampling}(i) and Theorem~\ref{thm:convgofSGD},
$\mathcal{A}_{SGD}$ applied to $\mathbb{P}_{N}(\lambda)$
has a sublinear convergence rate:
\[
\mathbb{E}^{\wp(\lambda)}\left[\left\Vert y_{t}-x^{*}(\lambda)\right\Vert _{2}^{2}\right]\leq\frac{(6L_{N}(\lambda)^{2}/\mu_{f}^{2}-1)\left\Vert y_{0}-x^{*}(\lambda)\right\Vert _{2}^{2}+8M_{B}^{2}/\mu_{f}^{2}}{t},\quad\forall\,t\geq1.
\]
\end{example}

\begin{example}
(SVRG, see \cite{johnson2013accelerating,xiao2014proximal}) SVRG
is an epoch-based algorithm, where each epoch consists of $M\geq1$
inner iterations. At the beginning of epoch $t\geq1$, we do a full
gradient evaluation at the current iterate $y_{t}$. Then, all of
the inner iterations of this epoch follow (using $l\geq0$ as the
index for the inner iterations within an epoch of SVRG):
\[
\tilde{y}_{l+1}=\Pi_{\mathcal{X}}\left[\tilde{y}_{l}-\gamma\left(\nabla f_{I_{l}}^{N}(\tilde{y}_{l},\lambda)-\nabla f_{I_{l}}^{N}(y_{t},\lambda)+\sum_{i\in[m]}\wp_{i}(\lambda)\nabla f_{i}^{N}(y_{t},\lambda)\right)\right],\,\forall\,l\in[M]_{0}.
\]
At the end of epoch $t$, we take $y_{t+1}=\tilde{y}_{M}$ and begin
the next epoch with a full gradient evaluation at $y_{t+1}$. By Remark
\ref{rem:special cases of sampling}(i) and Theorem \ref{thm:linear convergence rate of SVRG}, we pick the constant step size $\gamma_{*}=\frac{\mu_{f}}{(3+2M)L_{N}(\lambda)^{2}}$ 
to minimize the contraction factor and obtain the convergence rate 
across epochs in $\mathcal{A}_{SVRG}$:
\[
\mathbb{E}^{\wp(\lambda)}\left[\|y_{t}-x^{*}(\lambda)\|_{2}^{2}\right]\leq\left(1-\frac{\mu_{f}^{2}}{(3+2M)L_{N}(\lambda)^{2}}\right)\mathbb{E}^{\wp(\lambda)}\left[\|y_{t-1}-x^{*}(\lambda)\|_{2}^{2}\right],\,\forall\,t\geq1.
\]
\end{example}

Since the dual update is deterministic, all of the randomness in our
overall algorithm comes from using $\mathcal{A}$ to do the primal
updates. So, we only need to specify the required number of iterations
of $\mathcal{A}$ in each iteration to achieve our overall tolerable
error threshold. We formalize the notation for this sample complexity
in the following assumption, which directly corresponds to Assumption
\ref{assu:rate}.
\begin{assumption}
\label{assu:sample_complexity} Choose $K\geq1$, $\epsilon>0$, and
$\delta\in(0,\,1)$, then there exists $\{J_{k}(K,\epsilon,\delta)\}_{k\in[K]_{0}}$
such that: if $\mathcal{A}$ is run for $J\geq J_{k}(K,\epsilon,\delta)$
iterations for all $k\in[K]_{0}$, then $x^{k+1}$ satisfies Eq.~(\ref{eq:inexact primal update})
for Problem $\mathbb{P}_{N}(\lambda^{k})$ with probability at least
$(1-\delta)^{1/K}$ for all $k\in[K]_{0}$.
\end{assumption}

We can determine the complexity $\{J_{k}(K,\epsilon,\delta)\}_{k\in[K]_{0}}$
in Assumption \ref{assu:sample_complexity} for any specific $\mathcal{A}$ from Assumption \ref{assu:rate}.

\subsection{Convergence Analysis}

Our complexity analysis is based on the following intuition. In the $k$-th iteration, the goal is for $\mathcal{A}$ to return a nearly optimal solution to $\mathbb{P}_{N}(\lambda^{k})$ (based on our relaxed stopping condition) with probability at least $(1-\delta)^{1/K}$.
By the chain rule for the probability of the intersection of events,
the primal updates in all $K$ iterations will then be nearly optimal
with probability at least $1-\delta$.

We now provide the overall complexity of RanNLR, which is the total number of inner iterations (in $t$) across all outer iterations (in $k\in[K]_{0}$).
We emphasize that this convergence result is with respect to the distance to the unique optimal solution of Problem $\mathbb{P}$, i.e., $\|x^{K}-x^{*}\|_{\infty}$.
\begin{thm}
\label{thm: linear convergence for randomized LS} Suppose Assumptions
\ref{solandSlater}, \ref{assu:functions assump}, and \ref{regularity} hold. Let $\{x^k\}_{k \geq 0}$ be produced by Algorithm~\ref{alg: randomized nonlinear rescaling}. Choose $\varepsilon>0$ and $\delta\in(0,1)$. Then, there exists $N_{L}>0$ (independent of $\varepsilon$ and $\delta$) such that for all $N>\max\left\{ N_{L},c_{R}\right\}$,
$\epsilon=(1-c_{R}N^{-1})\varepsilon/(4C_{\Phi})$, and
\[
K=\left\lceil \frac{\ln\left(2\left\Vert \lambda^{0}-\lambda^{*}\right\Vert _{\infty}/\varepsilon\right)}{\ln\left(N/c_{R}\right)}\right\rceil ,
\]
we have $\left\Vert x^{K}-x^{*}\right\Vert _{\infty}\leq\varepsilon$ with
probability at least $1-\delta$.

(i) The overall complexity of Algorithm \ref{alg: randomized nonlinear rescaling} is $\tilde{O}\left(1/(\varepsilon^{2}\delta)\right)$ if $\mathcal{A}$ satisfies Assumption \ref{assu:rate}(i).

(ii) The overall complexity of Algorithm \ref{alg: randomized nonlinear rescaling}
is $O\left(\ln(1/\varepsilon)(2\ln(1/\varepsilon)+\ln(1/\delta))\right)$ if $\mathcal{A}$ satisfies Assumption~\ref{assu:rate}(ii).
\end{thm}

The following result is an immediate consequence of our bound on $\left\Vert x^{K}-x^{*}\right\Vert _{\infty}$ (it follows by the Lipschitz continuity of the objective and constraint functions).
\begin{cor}
Suppose Assumptions \ref{solandSlater}, \ref{assu:functions assump}, and
\ref{regularity} hold. Choose $\varepsilon>0$
and $\delta\in(0,1)$. Then, $x^{K}$ produced by Algorithm \ref{alg: randomized nonlinear rescaling}
satisfies $f(x^{K})-f(x^{*})\leq L_{f}^{(0)}\varepsilon$ (optimality
gap) and $g_{i}(x^{K})\leq L_{g}^{(0)}\varepsilon$ for all $i\in[m]$
(constraint violation), with probability at least $1-\delta$.
\end{cor}

\begin{rem}
If the scaling parameter $N$ is large, then the required $K$ will
be small. However, making $N$ larger also makes the condition number
of Problem $\mathbb{P}_{N}(\lambda)$ larger, which results in a slower
rate of convergence for the subroutine $\mathcal{A}$. For larger
$N$, the required number of outer iterations $K$ will be smaller,
but the required number of inner iterations for the subroutine $\mathcal{A}$
will be larger. Thus, there is a trade-off in the required number of outer iterations versus inner iterations through the selection of $N$.
\end{rem}

\subsection{Proof of Theorem \ref{thm: linear convergence for randomized LS}}

Given dual variables $\lambda\in\mathbb{R}_{++}^{m}$, the next primal-dual
pair $(\hat{x},\hat{\lambda})$ determined by inexact NLR is generated
by Eqs.~\eqref{eq:inexact primal update}-\eqref{eq:dual update}.
Lemma \ref{lem:one step NR} below on the one-step error is a modification of \cite[Proposition 1]{griva2006primal}
to account for our new stopping criterion Eq.~\eqref{eq:inexact primal update}.
\begin{lem}
\label{lem:one step NR} Suppose Assumptions \ref{solandSlater},
\ref{assu:functions assump}, and \ref{regularity} hold. For dual
variables $\lambda\in\mathbb{R}_{++}^{m}$ and sufficiently small
error tolerance $\epsilon>0$, there exists $N_{L}>0$ independent
of $\epsilon>0$, such that for all $N\geq N_{L}$, we have
\[
\max\left\{ \left\Vert \hat{x}-x^{*}\right\Vert _{\infty},\|\hat{\lambda}-\lambda^{*}\|_{\infty}\right\}  \leq2\,C_{\Phi}\epsilon+c_{R}N^{-1}\left\Vert \lambda-\lambda^{*}\right\Vert _{\infty}.
\]
\end{lem}

The next result shows that all primal iterates $\left\{ x^{k}\right\} _{k\geq1}$
of inexact NLR are bounded. 
\begin{lem}
\label{lem:bddness under inexact NR} Suppose Assumptions \ref{solandSlater},
\ref{assu:functions assump}, and \ref{regularity} hold. For dual
variables $\lambda^{0}\in\mathbb{R}_{++}^{m}$ and sufficiently small
error tolerance $\epsilon>0$, there exists $N_{L}>0$ independent
of $\epsilon>0$, such that for all $N\geq N_{L}$, we have
\[
\left\Vert x^{k}-x^{*}\right\Vert _{\infty}\leq2C_{\Phi}\epsilon/(1-c_{R}N^{-1})+c_{R}N^{-1}\left\Vert \lambda^{0}-\lambda^{*}\right\Vert _{\infty},\quad\forall\, k\geq1.
\]
\end{lem}

We can now provide the required number of iterations $\left\{ J_{k}(K,\epsilon,\delta)\right\} _{k\in[K]_{0}}$
of each call to $\mathcal{A}$ at different iterations of Algorithm
\ref{alg: randomized nonlinear rescaling}, so that Eq.~\eqref{eq:inexact primal update}
is satisfied with probability at least $(1-\delta)^{1/K}$.

We first consider the case where $\mathcal{A}$ has a sublinear convergence
rate. Define $J_{0}^{SL}(K,\epsilon,\delta)$ as
\begin{align}
 \max\left\{ \left\lceil \frac{L_{N}(\lambda^{0})^{2}\left(nA\left(\left\Vert x^{0}-x^{*}\right\Vert _{\infty}^{2}+(c_{R}N^{-1})^{2}\left\Vert \lambda^{0}-\lambda^{*}\right\Vert _{\infty}^{2}\right)+B\right)}{\left(1-(1-\delta)^{1/K}\right)\epsilon^{2}}\right\rceil ,1\right\} ,\label{eq:J0SGD}
\end{align}
and define $J_{k}^{SL}(K,\epsilon,\delta)$ as
\begin{align}
\max\left\{ \left\lceil \frac{L_{N}(\lambda^{k})^{2}\left(nA\left(1+(c_{R}N^{-1})^{2}\right)\left(2C_{\Phi}\epsilon/(1-c_{R}N^{-1})+(c_{R}N^{-1})^{k}\left\Vert \lambda^{0}-\lambda^{*}\right\Vert _{\infty}\right)^{2}+B\right)}{\left(1-(1-\delta)^{1/K}\right)\epsilon^{2}}\right\rceil ,1\right\} ,\label{eq:JkSGD}
\end{align}
for each $k\in[K-1]$. 
%We see that $J_{k}^{SL}(K,\epsilon,\delta)=O(K/(\epsilon^{2}\delta))$ for each $k\in[K]_{0}$.

\begin{lem}
\label{lem:complexity_SGD} Suppose $\mathcal{A}$ satisfies Assumption
\ref{assu:rate}(i).

(i) At iteration $k=0$, if we run $\mathcal{A}$ for $J_{0}^{SL}(K,\epsilon,\delta)$
iterations, then the output $x^{1}$ satisfies Eq.~(\ref{eq:inexact primal update})
with probability at least $(1-\delta)^{1/K}$.

(ii) At iteration $k\geq1$, suppose that the solution returned
by $\mathcal{A}$ satisfies Eq.~(\ref{eq:inexact primal update})
in all previous iterations. If we run $\mathcal{A}$ for \textup{$J_k^{SL}(K,\epsilon,\delta)$}
iterations, then the output\textbf{ $x^{k+1}$ }satisfies Eq.~(\ref{eq:inexact primal update})
with probability at least $(1-\delta)^{1/K}$.
\end{lem}

Now we treat the case where $\mathcal{A}$ has a linear convergence
rate. Define $J_{0}^{L}(K,\epsilon,\delta)$ as
\begin{equation}
\max\left\{ \left\lceil \ln\left(\frac{n\zeta L_{N}(\lambda^{0})^{2}\left(\left\Vert x^{0}-x^{*}\right\Vert _{\infty}^{2}+(c_{R}N^{-1})^{2}\left\Vert \lambda^{0}-\lambda^{*}\right\Vert _{\infty}^{2}\right)}{\left(1-(1-\delta)^{1/K}\right)\epsilon^{2}}\right)/\ln(1/\alpha)\right\rceil ,1\right\} ,\label{eq:J0SVRG}
\end{equation}
and define $J_{k}^{L}(K,\epsilon,\delta)$ as
\begin{align}
\max\left\{ \left\lceil \ln\left(\frac{n\zeta L_{N}(\lambda^{k})^{2}\left(1+(c_{R}N^{-1})^{2}\right)\left(2C_{\Phi}\epsilon/(1-c_{R}N^{-1})+(c_{R}N^{-1})^{k}\left\Vert \lambda^{0}-\lambda^{*}\right\Vert _{\infty}\right)^{2}}{\left(1-(1-\delta)^{1/K}\right)\epsilon^{2}}\right)/\ln(1/\alpha)\right\rceil ,1\right\} ,\label{eq:JkSVRG}
\end{align}
for each $k\in[K-1]$. 
%We note that $J_{k}^{L}(K,\epsilon,\delta)=O\left(\ln(K/(\epsilon^{2}\delta))\right)$ for each $k\in[K]_{0}$. 
Using the same argument as Lemma \ref{lem:complexity_SGD},
we have the following complexity result for the present case.
\begin{lem}
\label{lem:complexity-SVRG} Suppose $\mathcal{A}$ satisfies Assumption
\ref{assu:rate}(ii).

(i) At iteration $k=0$, if we run $\mathcal{A}$ with $J_{0}^{L}(K,\epsilon,\delta)$
iterations, then the output $x^{1}$ satisfies Eq.~(\ref{eq:inexact primal update})
with probability at least $(1-\delta)^{1/K}$.

(ii) At iteration $k\geq1$, suppose that the solution returned
by $\mathcal{A}$ satisfies Eq.~(\ref{eq:inexact primal update})
in all previous iterations. If we run $\mathcal{A}$ for \textup{$J_k^{L}(K,\epsilon,\delta)$}
iterations, then the output\textbf{ $x^{k+1}$ }satisfies Eq.~(\ref{eq:inexact primal update})
with probability at least $(1-\delta)^{1/K}$.
\end{lem}

Based on Lemmas \ref{lem:complexity_SGD} and \ref{lem:complexity-SVRG},
we can now provide the proof of Theorem \ref{thm: linear convergence for randomized LS}.
If Eq.~\eqref{eq:inexact primal update} holds for all iterations
$k\in[K]_{0}$, then iterating Lemma \ref{lem:one step NR}
shows:
\[
\left\Vert x^{K}-x^{*}\right\Vert _{\infty}\leq\left(1-(c_{R}N^{-1})^{K}\right)2C_{\Phi}\epsilon/(1-c_{R}N^{-1})+(c_{R}N^{-1})^{K}\left\Vert \lambda^{0}-\lambda^{*}\right\Vert _{\infty}.
\]
Plugging in $\epsilon=(1-c_{R}N^{-1})\varepsilon/(4C_{\Phi})$ and
$K=\left\lceil \ln\left(2\left\Vert \lambda^{0}-\lambda^{*}\right\Vert _{\infty}/\varepsilon\right)/\ln\left(N/c_{R}\right)\right\rceil =O\left(\ln(1/\varepsilon)\right)$
into the above inequality, and noting that $1-(c_{R}N^{-1})^{K}<1$,
we see that $\left\Vert x^{K}-x^{*}\right\Vert _{\infty}\leq\varepsilon$.

We estimate the complexity of Algorithm \ref{alg: randomized nonlinear rescaling}
as follows. First suppose $\mathcal{A}$ has a sublinear convergence
rate. By Lemma \ref{lem:complexity_SGD}, we know that if $\mathcal{A}$
is run for $J_{k}^{SL}(K,(1-c_{R}N^{-1})\varepsilon/(4C_{\Phi}),\delta)$
iterations, then it returns an approximately optimal solution satisfying
Eq.~\eqref{eq:inexact primal update} with $\epsilon=(1-c_{R}N^{-1})\varepsilon/(4C_{\Phi})$
with probability at least $(1-\delta)^{1/K}$. Using the chain rule
for the probability of the intersection of events, the solutions $x^{k}$
($k\in[K]$) returned from $\mathcal{A}$ satisfy Eq.~\eqref{eq:inexact primal update}
with $\epsilon=(1-c_{R}N^{-1})\varepsilon/(4C_{\Phi})$ for all $K$
iterations with probability at least $1-\delta$. Clearly, $J_{k}^{SL}(K,(1-c_{R}N^{-1})\varepsilon/(4C_{\Phi}),\delta)=O(K/(\varepsilon^{2}\delta))$
for all $k\in[K]_{0}$ and $K=O\left(\ln(1/\varepsilon)\right)$.
Therefore, the overall complexity of Algorithm \ref{alg: randomized nonlinear rescaling}
is $O\left(\ln^{2}(1/\varepsilon)(1/(\varepsilon^{2}\delta))\right)$.

Now suppose that $\mathcal{A}$ has a linear convergence
rate. Clearly, $J_{k}^{L}(K,(1-c_{R}N^{-1})\varepsilon/(4C_{\Phi}),\delta)=O\left(\ln(K/(\varepsilon^{2}\delta))\right)$
for all $k\in[K]_{0}$ and $K=O\left(\ln(1/\varepsilon)\right)$.
Using the same argument as above along with Lemma \ref{lem:complexity-SVRG},
we see that the overall complexity of Algorithm \ref{alg: randomized nonlinear rescaling}
is $O\left(\ln(1/\varepsilon)(2\ln(1/\varepsilon)+\ln(1/\delta))\right)$.

\subsection{Optimal Sampling from $\wp(\lambda)$ \label{sec:Advantage-of-adaptive}}

In our implementation of RanNLR, the subroutine $\mathcal{A}$ samples from $\wp(\lambda) \in \mathcal P_{+}([m])$ when solving each instance of Problem $\mathbb{P}_{N}(\lambda)$. For the purpose of comparison, let $\mathcal Q = (q_i)_{i \in [m]}$ be any other probability distribution in $\mathcal P_{+}([m])$, and let $\mathbb{E}^{\mathcal{Q}}\left[\cdot\right]$
denote conditional expectation with respect to $\mathcal{Q}$. We use $\mathcal{A}^{\mathcal Q}$ to denote the subroutine $\mathcal{A}$ where $\mathcal{Q}$ is used to sample $\{I_t\}_{t \geq 0}$ instead of $\wp(\lambda)$.
In particular, we denote $\mathcal{A}_{SGD}^{\mathcal Q}$ and $\mathcal{A}_{SVRG}^{\mathcal Q}$ as SGD and SVRG with a generic sampling distribution $\mathcal{Q}$, respectively.

We want to compare the performance of $\mathcal A$ (which samples from $\wp(\lambda)$) with $\mathcal{A}^{\mathcal Q}$. The key difference in performance between $\mathcal A$ and $\mathcal{A}^{\mathcal Q}$ is captured by the constant $r^{\mathcal{Q}}\triangleq\sum_{i\in[m]}\wp_{i}(\lambda)^{2}/q_{i}$.

\begin{lem}
Fix $\lambda\in\mathbb{R}_{++}^{m}$ and let $\mathcal{Q}\triangleq(q_{i})_{i\in[m]}\in \mathcal P_{+}([m])$. Then $r^{\mathcal{Q}}\geq1$,
and equality holds if and only if $q_{i}=\wp_{i}(\lambda)$ for all $i\in[m]$.
\end{lem}

\begin{proof}
Define the random variable $X$ such that $\text{Pr}\left(X = \wp_{i}(\lambda)/q_{i}\right) = q_{i}$, for all $i\in[m]$.
By Jensen's inequality, we have
\[
r^\mathcal Q = \sum_{i \in [m]} q_i\left(\wp_{i}(\lambda)/q_{i}\right)^2=\mathbb{E}[X^{2}] \geq \left(\mathbb{E}[X]\right)^2 = 1,
\]
since $\mathbb E[X] = \sum_{i \in [m]} q_i \wp_{i}(\lambda)/q_{i} = 1$. Equality holds if and only if $q_{i}=\wp_{i}(\lambda)$ for all $i\in[m]$.
\end{proof}

When solving Problem $\mathbb{P}_{N}(\lambda)$, it is optimal to sample from $\wp(\lambda)$ for both SGD and SVRG.

\begin{example}
(SGD with generic sampling) Fix $\lambda\in\mathbb{R}_{++}^{m}$ and let $M_{B}\triangleq\max_{i\in[m]}\left\Vert \nabla f_{i}^{N}(x^{*}(\lambda),\lambda)\right\Vert _{2}$.
By Theorem \ref{thm:convgofSGD}, the convergence rate of $\mathcal{A}_{SGD}^{\mathcal Q}$ applied to $\mathbb{P}_{N}(\lambda)$
is:
\begin{equation}
\mathbb{E}^{\mathcal{Q}}\left[\left\Vert y_{t}-x^{*}(\lambda)\right\Vert _{2}^{2}\right]\leq(A_{SGD}^{\mathcal{Q}}\left\Vert y_{0}-x^{*}(\lambda)\right\Vert _{2}^{2}+B_{SGD}^{\mathcal{Q}})/t,\quad\forall\,t\geq1,\label{SGD-adaptive}
\end{equation}
where $A_{SGD}^{\mathcal{Q}}\triangleq 2(1+2r^{\mathcal{Q}})L_{N}(\lambda)^{2}/\mu_{f}^{2}-1$
and $B_{SGD}^{\mathcal{Q}}\triangleq8r^{\mathcal{Q}}M_{B}^{2}/\mu_{f}^{2}$.
For comparison, the convergence rate of $\mathcal{A}_{SGD}$ satisfies
Eq.~(\ref{SGD-adaptive}) with $A_{SGD}\triangleq6L_{N}(\lambda)^{2}/\mu_{f}^{2}-1$
and $B_{SGD}\triangleq 8M_{B}^{2}/\mu_{f}^{2}$. Since $r^{\mathcal{Q}}\geq1$,
we must have $A_{SGD}^{\mathcal{Q}}\geq A_{SGD}$ and
$B_{SGD}^{\mathcal{Q}}\geq B_{SGD}$ (for any other $\mathcal Q$). Thus, if we
use any distribution other than $\wp\left(\lambda\right)$ to sample $\left\{ I_{t}\right\} _{t\geq0}$, then the constants in the sublinear rate will increase (and become worse).
\end{example}

\begin{example}
(SVRG with generic sampling) Fix $\lambda\in\mathbb{R}_{++}^{m}$ and a constant step-size $\gamma_{*}=\frac{\mu_{f}}{(1+2r+2Mr)L_{N}(\lambda)^{2}}$. By Theorem \ref{thm:linear convergence rate of SVRG},
the convergence rate of $\mathcal{A}_{SVRG}^{\mathcal Q}$ applied to $\mathbb{P}_{N}(\lambda)$ is:
\begin{equation}
\mathbb{\mathbb{E}^{\mathcal{Q}}}\left[\|y_{t}-x^{*}(\lambda)\|_{2}^{2}\right]\leq\left(\alpha_{SVRG}^{\mathcal{Q}}\right)^{t}\mathbb{\mathbb{E}^{\mathcal{Q}}}\left[\|y_{0}-x^{*}(\lambda)\|_{2}^{2}\right],\label{SVRG-adaptive}
\end{equation}
where $\alpha_{SVRG}^{\mathcal{Q}}\triangleq1-\frac{\mu_{f}^{2}}{(1+2r^{\mathcal{Q}}+2Mr^{\mathcal{Q}})L_{N}(\lambda)^{2}}$.
For comparison, the convergence rate of $\mathcal{A}_{SVRG}$ satisfies Eq.~(\ref{SVRG-adaptive})
with contraction factor $\alpha_{SVRG}\triangleq 1-\frac{\mu_{f}^{2}}{(3+2M)L_{N}(\lambda)^{2}}$.
Since $r^{\mathcal{Q}}\geq1$, we must have $\alpha_{SVRG}^{\mathcal{Q}}\geq\alpha_{SVRG}$.
Thus, if we use any distribution other than $\wp\left(\lambda\right)$
to sample $\left\{ I_{t}\right\} _{t\geq0}$, then the
contraction factor will increase (and become worse).
\end{example}

\section{\label{sec:Numerical-Experiments} Numerical Experiments }

In this section, we present two case studies to illustrate the effectiveness
and behavior of RanNLR. The first one is a simple case adapted from
\cite{mehrotra2014cutting}, where we choose the primal-dual type
algorithm from \cite{yu2016primal} as the baseline for comparison.
This baseline algorithm theoretically achieves an $O(1/K)$ convergence
rate in terms of the optimality gap and constraint violation. The
second one is an inventory control problem adapted from \cite{lin2019revisiting},
where we compare our algorithm with a commercial solver (Gurobi 9.0)
in the task of solving an LP with one million constraints.

We use the following nonlinear rescaling function:
\begin{equation}
\psi(t)\triangleq\begin{cases}
1-e^{-t}, & t\geq-0.5,\\
-0.5e^{0.5}t^{2}+0.5e^{0.5}t+1-\frac{5}{8}e^{0.5}, & t\leq-0.5,
\end{cases}\label{eq:nonlinear_rescaling_fun}
\end{equation}
for our implementation of RanNLR.
\begin{rem}
In the NLR literature, Newton's method is used to do the primal update
\cite{griva2006primal,polyak2004primal}. The corresponding numerical
results demonstrate the ``hot start'' phenomenon, where only a few
updates and very few (often just one) Newton steps per update are
required. It can be expected that our algorithm will also experience this ``hot start'' phenomenon. After a few updates, the optimizer obtained
from $\mathcal{A}$ is always in the neighborhood of the next one,
and so we require fewer and fewer iterations of $\mathcal{A}$ to
do the primal updates with an extra digit of accuracy.
\end{rem}

\begin{rem}
It is hard to use a universal nonlinear rescaling function for all
problems. Therefore, we may need to tune the rescaling function for different
problems. However, as we will see in Subsection \ref{ALP case}, we
can simply normalize Eq.~\eqref{eq:nonlinear_rescaling_fun} with division by a positive number $\beta$ to control the range of the values of the constraint functions. In this way, we can just tune $\beta$ rather than redesign the entire nonlinear rescaling function.
\end{rem}

\subsection{A Simple Case}

This case is adapted from a semi-infinite programming problem in \cite{mehrotra2014cutting},
where we discretize the constraint index set over a uniform grid and construct
the following approximate problem:
\begin{alignat}{1}
\min\  & \left(x_{1}-2\right)^{2}+\left(x_{2}-0.2\right)^{2}\nonumber \\
\text{s.t. } & \left(\frac{5\sin\left(\pi\sqrt{\frac{i}{m}}\right)}{1+\left(\frac{i}{m}\right)^{2}}\right)x_{1}^{2}-x_{2}\le0,\ \ \forall i\in[m],\label{eq:constraint}\\
 & x_{1}\in\left[-1,1\right],x_{2}\in\left[0,0.2\right].\nonumber
\end{alignat}
The optimal solution is $x=(0.20523677,0.2)$ and the optimal value
is $3.221$. We set $m=10,000$ and compare the performance of RanNLR
(with SVRG as the subroutine, henceforth denoted RanNLR-$\mathcal{A}_{\text{SVRG}}$)
with the baseline algorithm taken from \cite{yu2016primal}. The results
are presented in Table \ref{tab:Simulation-results} and Figure \ref{fig:Simulation-results_fig}.

\begin{table}
\caption{\label{tab:Simulation-results}Simulation results}

\centering{}%
\begin{tabular}{|c|c|c|c|c|c|c|c|c|}
\hline
{\small{}{}Algorithm}  & {\small{}{}Stepsize $\gamma$}  & {\small{}{}Iteration $K$}  & {\small{}{}$N$}  & {\small{}{}Epoch no. $M$}  & {\small{}{}$\epsilon$}  & {\small{}{}Obj.}  & {\small{}{}Relative gap}  & {\small{}{}CPU time (s)}\tabularnewline
\hline
\hline
{\small{}{}Baseline}  & {\small{}{}$0.0001$}  & {\small{}{}$30000$}  & {\small{}{}-}  & {\small{}{}-}  & {\small{}{}-}  & {\small{}{}$3.231$}  & {\small{}{}$0.3\%$}  & {\small{}{}$2.07$}\tabularnewline
\hline
{\small{}{}PDSVRG}  & {\small{}{}$0.0001$}  & {\small{}{}$62$}  & {\small{}{}$100$}  & {\small{}{}$20$}  & {\small{}{}$0.0001$}  & {\small{}{}$3.221$}  & {\small{}{}$<0.01\%$}  & {\small{}{}$0.10$}\tabularnewline
\hline
{\small{}{}PDSVRG}  & {\small{}{}$0.0001$}  & {\small{}{}$4$}  & {\small{}{}$1000$}  & {\small{}{}$400$}  & {\small{}{}$0.0001$}  & {\small{}{}$3.221$}  & {\small{}{}$<0.01\%$}  & {\small{}{}$0.02$}\tabularnewline
\hline
\end{tabular}
\end{table}
\begin{figure}
\begin{centering}
\includegraphics[scale=0.6]{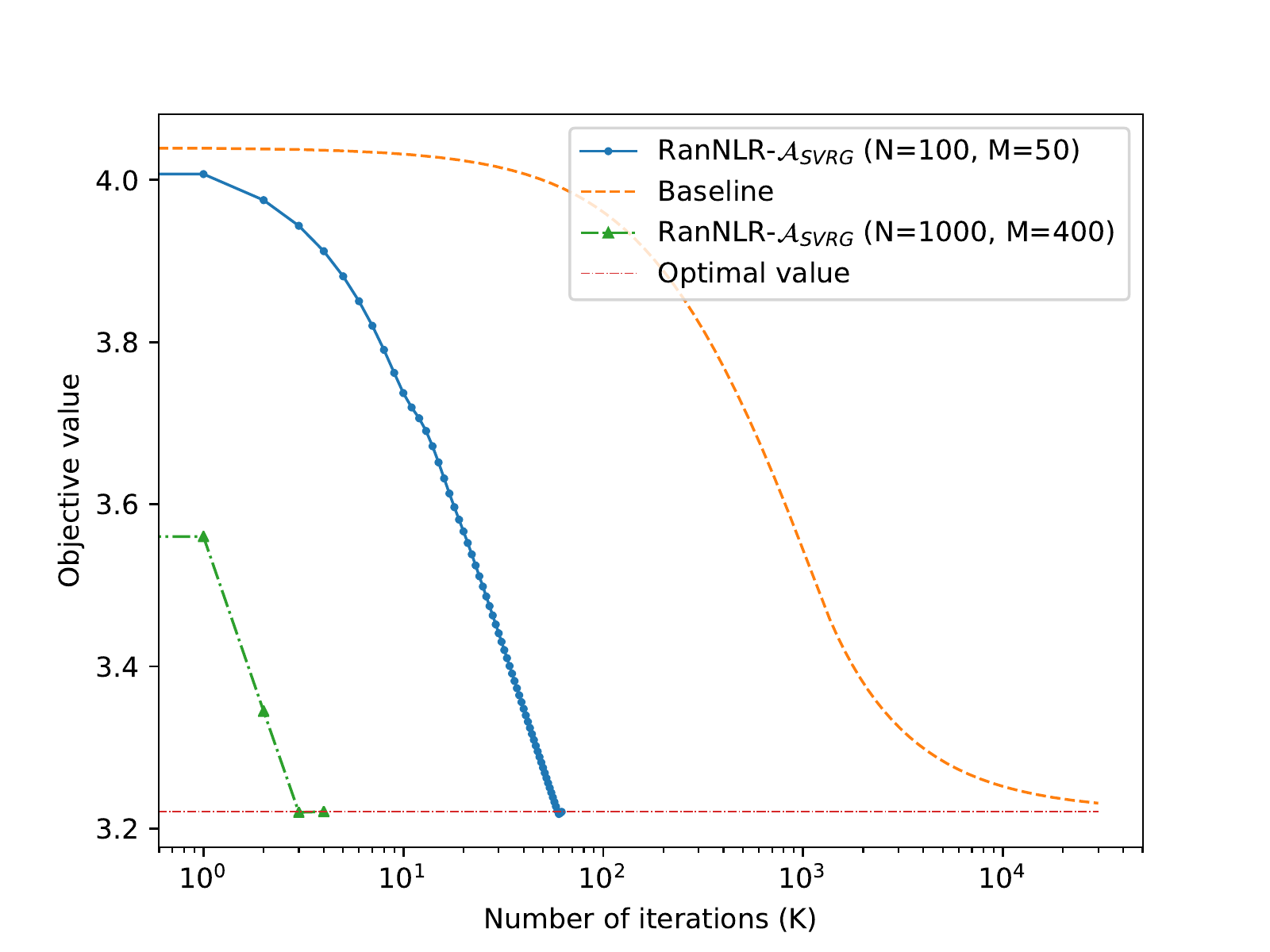}
\par\end{centering}
\caption{\label{fig:Simulation-results_fig}Convergence analysis}
\end{figure}
We run the baseline algorithm for 30,000 iterations and see the achieved
objective value is $3.231$, with a relative optimality gap of
$0.3\%$. However, RanNLR-$\mathcal{A}_{\text{SVRG}}$ converges to
the optimal value $3.221$ within $62$ iterations, with a relative
optimality gap of only $0.01\%$. In addition, the CPU time of RanNLR-$\mathcal{A}_{\text{SVRG}}$
($0.10$ seconds) is faster than the baseline algorithm ($2.07$ seconds).
We also explore the effect of changing the scaling parameter $N$.
When $N=1,000$ and the number of epochs used in SVRG is $400$, the
algorithm converges to a nearly optimal solution in merely 4 iterations and the required CPU time decreases to $0.02$ seconds.

\subsection{Inventory Control with Approximate Linear Programming}

\label{ALP case}

In this subsection, we present a numerical study adapted from \cite{lin2019revisiting}
for a single-product inventory control problem with partially backlogged
demand and zero lead time. Let $s\in\mathbb{S}\subset\left[-10,10\right]$
and $a\in\mathbb{A}\subset\left[0,20\right]$ denote the on-hand inventory
(with negative values indicating backlogged orders) and the order
quantity, respectively. We discretize the state/action space with
a precision of $0.02$ so that $\mathbb{S}\triangleq\left\{ -10,-9.98,\ldots,9.98,10\right\} $
and $\mathbb{A}\triangleq\left\{ 0,0.02,\ldots,19.98,20\right\} $.
Let $D$ denote the stochastic customer demand, which is assumed to have
a discrete sample space $\mathbb{D}\triangleq\{0,0.02,\ldots,9.98,10\}$.
The probability density for $D$ is given by $p(d)=\int_{d-0.01}^{d+0.01}f_{D}(x)dx$
for each $d\in\mathbb{D}$, where $f_{D}(\cdot)$ is the density function
of a truncated normal distribution on $\left[0,10\right]$ with mean
$5$ and standard deviation $2$. The state transitions are given
by $s^{\prime}\triangleq\min\left(\max\left(s+a-D,\underline{l}\right),\bar{u}\right)$,
where $\underline{l}=-10$ and $\bar{u}=10$ denote the lower/upper
bounds of the state, respectively. The cost function for each state-action
pair $\left(s,a\right)\in\mathbb{S}\times\mathbb{A}$ is
\[
c\left(s,a\right)\triangleq c_{p}a+c_{h}\mathbb{E}\left[\left(s^{\prime}\right)_{+}\right]+c_{b}\mathbb{E}\left[\left(-s^{\prime}\right)_{+}\right]+c_{d}\mathbb{E}\left[\left(s+a-D-\bar{u}\right)_{+}\right]+c_{l}\mathbb{E}\left[\left(\underline{l}-s-a+D\right)_{+}\right],
\]
where $\left(c_{p},c_{h},c_{b},c_{d},c_{l}\right)=\left(20,2,10,10,100\right)$
are the cost coefficients. The discount factor is $\gamma = 0.95$ in this example. 

Let $V:\mathbb{S}\rightarrow\mathbb{R}$ be the value function of
the MDP, and let $q\left(\cdot\right)$ be a given probability mass function
on $\mathbb{S}$ (in the numerical study, we choose $q\left(\cdot\right)$
to be uniform on $\mathbb{S}$). The above infinite time horizon MDP can be solved
by the following linear programming problem:
\begin{alignat}{1}
\text{\ensuremath{\max_{V}\ }} & \mathbb{E}_{q}\left[V\left(s\right)\right]\label{eq:LP_formulation_MDP}\\
\text{s.t.\ } & V\left(s\right)-\gamma\mathbb{E}_{D}\left[\left.V\left(s^{\prime}\right)\right|s,a\right]\le c\left(s,a\right),\ \ \forall\left(s,a\right)\in\mathbb{S}\times\mathbb{A}.\nonumber
\end{alignat}
The above problem optimizes over $V\left(\cdot\right)$, which is
an intractable problem when the state space is large or infinite.

Alternatively, a more tractable method approximates the value function
$V$ via a linear combination of basis functions $\phi_{b}\left(\cdot\right)$,
i.e., the approximate linear programming (ALP) method. The ALP formulation
of the above problem is:
\begin{alignat}{1}
\text{\ensuremath{\max_{\theta}\ }} & \sum_{b=1}^{B}\theta_{b}\mathbb{E}_{q}\left[\phi_{b}\left(s\right)\right]\label{eq:ALP_formulation_MDP}\\
\text{s.t.\ } & \sum_{b=1}^{B}\theta_{b}\left(\phi_{b}\left(s\right)-\gamma\mathbb{E}_{D}\left[\left.\phi_{b}\left(s^{\prime}\right)\right|s,a\right]\right)-c\left(s,a\right)\le0,\ \ \forall\left(s,a\right)\in\mathbb{S}\times\mathbb{A},\nonumber
\end{alignat}
where $\phi_{b}\left(\cdot\right)$ are
the basis functions chosen to approximate the value functions of the
MDP, and $\theta_{b}$ are the weights of the basic functions. Problem \eqref{eq:ALP_formulation_MDP} is a large-scale linear
programming problem with $\left|\mathbb{S}\right|\times\left|\mathbb{A}\right|=1,002,001$
constraints (equal to the number of state-action pairs). We choose
$B=2$ and $\left(\phi_{1},\phi_{2}\right)=\left(1,s\right)$ for
our numerical study.

We first solve Problem \eqref{eq:ALP_formulation_MDP} via RanNLR
with SGD (henceforth denoted RanNLR-$\mathcal{A}_{\text{SGD}}$).
We normalize the constraints of Problem \eqref{eq:ALP_formulation_MDP}
by dividing the constraint functions by $\beta>0$ to get:
\[
\left(\sum_{b=1}^{B}\theta_{b}\left(\phi_{b}\left(s\right)-\gamma\mathbb{E}_{D}\left[\left.\phi_{b}\left(s^{\prime}\right)\right|s,a\right]\right)-c\left(s,a\right)\right)\big/\beta\le0,\ \ \forall\left(s,a\right)\in\mathbb{S}\times\mathbb{A},
\]
so that the values of the constraints range from $\left[-1,1\right]$.
We take $\beta=600$ for this problem instance. Then, we transform the
constraints via the nonlinear rescaling function Eq. \eqref{eq:nonlinear_rescaling_fun}.
In addition, we check the termination condition for SGD every $1000$
iterations (instead of every iteration) because calculating $\left\Vert \nabla_{x}\mathcal{L}_{N}(x^{k+1},\lambda^{k})\right\Vert _{\infty}$
can be expensive when the number of constraints is large.

Note that we need to calculate the expectations $\mathbb{E}_{D}[\cdot]$
for all $(s,a)\in\mathbb{S}\times\mathbb{A}$ in the constraints of
Problem \eqref{eq:ALP_formulation_MDP}, which can be time-consuming
if we calculate them on-the-fly. Instead, we calculate the expectations
and store the values in a matrix, and call the values whenever we
need them. Therefore, the CPU time listed in Table \ref{tab:Simulation-results-alp}
excludes the calculation time of $\mathbb{E}_{D}[\cdot]$ for all
three algorithms.

The optimal value of Problem \eqref{eq:ALP_formulation_MDP} is $2146.94$,
as provided by Gurobi 9.0. The simulation results for this case are
presented in Table \ref{tab:Simulation-results-alp}. We choose $N=1000$
and $\epsilon=1$. We see that RanNLR-$\mathcal{A}_{\text{SGD}}$
finds a near-optimal solution in $30$ outer iterations, and the number
of required iterations for the subroutine $\mathcal{A}_{\text{SGD}}$
ranges from $1000$ to $9000$. In addition, the CPU time for RanNLR-$\mathcal{A}_{\text{SGD}}$
is $11.0$ seconds, while the commercial solver takes $79.3$ seconds.

We also test the performance of RanNLR-$\mathcal{A}_{\text{SVRG}}$
on this problem instance. We see from Figure \ref{fig:Simulation-results_alp}
that the convergence rates of RanNLR-$\mathcal{A}_{\text{SVRG}}$
and RanNLR-$\mathcal{A}_{\text{SGD}}$ are similar. However, RanNLR-$\mathcal{A}_{\text{SVRG}}$
is more expensive per iteration and so its CPU time is longer compared
to that of RanNLR-$\mathcal{A}_{\text{SGD}}$.

\begin{table}
\caption{\label{tab:Simulation-results-alp}Simulation results with $N=1000$}

\centering{}{\small{}{}}%
\begin{tabular}{|c|c|c|c|c|c|c|c|c|}
\hline
{\small{}{}Algorithm}  & {\small{}{}Stepsize $\gamma$}  & {\small{}{}$K$}  & {\small{}{}$N$}  & {\small{}{}Epoch no. $M$}  & {\small{}{}$\epsilon$}  & {\small{}{}Obj.}  & {\small{}{}Relative gap}  & {\small{}{}CPU time (s)}\tabularnewline
\hline
\hline
{\small{}{}Solver}  & {\small{}{}-}  & {\small{}{}-}  & {\small{}{}-}  & {\small{}{}-}  & {\small{}{}-}  & {\small{}{}$2146.94$}  & {\small{}{}-}  & {\small{}{}$79.3$}\tabularnewline
\hline
RanNLR-$\mathcal{A}_{\text{SVRG}}$  & {\small{}{}$0.005$}  & {\small{}{}$30$}  & {\small{}{}$1000$}  & {\small{}{}$1000$}  & {\small{}{}$1$}  & {\small{}{}$2147.16$}  & {\small{}{}$0.010\%$}  & {\small{}{}$21.9$}\tabularnewline
\hline
RanNLR-$\mathcal{A}_{\text{SGD}}$  & {\small{}{}$0.005$}  & {\small{}{}$30$}  & {\small{}{}$1000$}  & {\small{}{}-}  & {\small{}{}$1$}  & {\small{}{}$2147.17$}  & {\small{}{}$0.011\%$}  & {\small{}{}$11.0$}\tabularnewline
\hline
\end{tabular}
\end{table}
\begin{figure}
\begin{centering}
\includegraphics[scale=0.6]{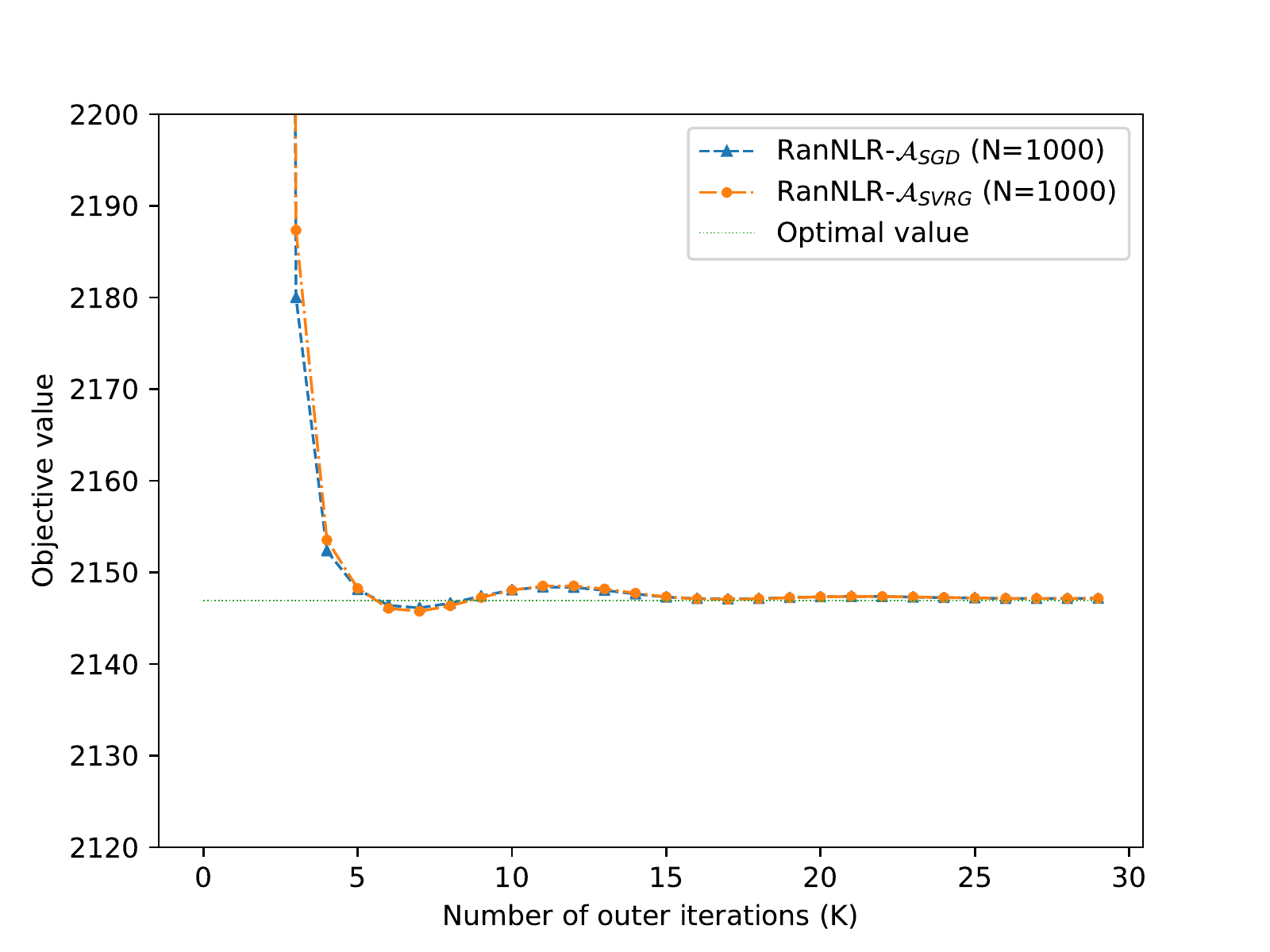}
\par\end{centering}
\caption{\label{fig:Simulation-results_alp}Convergence analysis with $N=1000$
and $\epsilon=1$}
\end{figure}
We plot the probability distribution $\wp\left(\lambda\right)$ for
the $5^{\text{th}}$ and $80^{\text{th}}$ outer iterations of RanNLR-$\mathcal{A}_{\text{SGD}}$
for $\left(N=1000\right)$ in Figure \ref{fig:alp_adaptive_sampling}.
This probability distribution is flat in the $5^{\text{th}}$ iteration,
but it concentrates on a certain subset of the constraints as the
algorithm proceeds. This phenomenon may help explain why our algorithm
can converge to a near-optimal solution by using several thousand
samples when there are more than one million constraints.

\begin{figure}
\centering{}\includegraphics[scale=0.6]{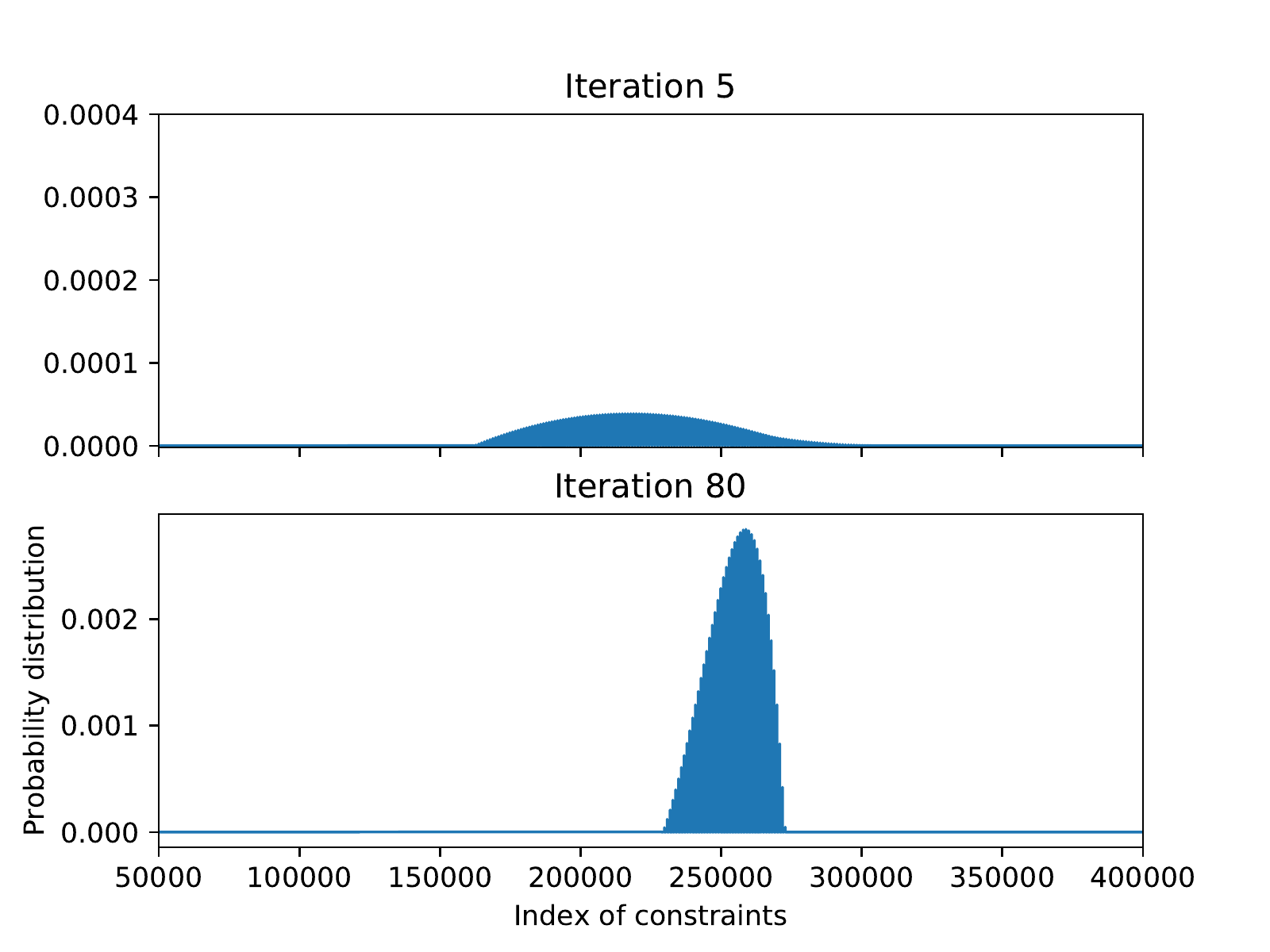}\caption{\label{fig:alp_adaptive_sampling}Probability distribution with adaptive
sampling}
\end{figure}

\section{\label{sec:Conclusion} Conclusion}

In this paper, we develop RanNLR to solve convex optimization programs
(with high probability) when the number of constraints is very large.
For a tolerance $\varepsilon>0$ and an overall failure probability
$\delta\in(0,1)$, we provide the complexity analysis for obtaining
a solution within $\varepsilon$ of the optimal solution to Problem
$\mathbb{P}$ with probability at least $1-\delta$. The core of RanNLR
is the use of randomized first-order algorithms to do the primal updates.
Due to the special structure of the finite sum minimization problem
in the primal updates, we can leverage on the success of these randomized
first-order algorithms to achieve significant computational savings
as suggested by our experiments.

We briefly remark on the connection between our results and the most
closely related works. In \cite{griva2006primal,polyak2004primal},
Newton's method is used to do the primal update. In this case, Hessian computations and matrix inversions are required in every
iteration, and the associated computational and storage requirements
are between $O((n+m)^{2})$ and $O((n+m)^{3})$, which is prohibitive
when $m \gg 0$. In addition,
the methods in \cite{yu2016primal} and \cite{yu2017simple} are based
on the average of the sequence of primal iterates and the convergence
rate is $O(1/K)$ for both the optimality gap and the constraint violation.
In contrast, RanNLR only uses the last primal iterate.

In future research, we will extend RanNLR to semi-infinite programming
problems. Here we will identify a tractable family of sampling distributions, and combine this family with the efficient
randomized first-order algorithms that have been demonstrated here.

\section*{Acknowledgement} 
This research is supported by Institute of Operations Research and Analytics grant (R-726-000-004-646) and Ministry of Education Tier 2 grant (R-263-000-C83-112).

\bibliographystyle{plain}
\bibliography{References}

\appendix
%dummy comment inserted by tex2lyx to ensure that this paragraph is not empty

\section{Proofs of Technical Lemmas}

\subsection{Proofs of Lemmas \ref{lem:one step NR} and \ref{lem:bddness under inexact NR}}

Recall that
\[
\Phi_{N}(x^{*},\lambda^{*})\triangleq\left[\begin{array}{cc}
\nabla_{xx}L(x^{*},\lambda^{*}), & -\nabla G_{I^{*}}(x^{*})^{T}\\
-\left\langle \varLambda_{I^{*}}^{*},\nabla G_{I^{*}}(x^{*})\right\rangle , & \psi''(0)^{-1}N^{-1}\mathrm{I}_{\left|I^{*}\right|}
\end{array}\right].
\]
Without loss of generality, assume that the active constraints $I^{*}$
correspond to the first $|I^{*}|$ components of $\lambda$. Then
define the following vectors: $\Delta\lambda\triangleq\hat{\lambda}-\lambda^{*}=\left(\Delta\lambda_{I^{*}},\Delta\lambda_{[m]\setminus I^{*}}\right)$,
$\Delta x\triangleq\hat{x}-x^{*}$, $\Delta z_{I^{*}}\triangleq\left(\Delta x,\Delta\lambda_{I^{*}}\right)$,
and $\Delta z\triangleq\left(\Delta x,\Delta\lambda\right)$. Moreover,
we define a mapping $h_{N}:\mathbb{R}^{n}\times\mathbb{R}_{++}^{m-\left|I^{*}\right|}\rightarrow\mathbb{R}$
as $h_{N}(x,\lambda_{[m]\setminus I^{*}})\triangleq\sum_{i\in[m]\setminus I^{*}}\lambda_{i}\psi'(Ng_{i}(x))\nabla g_{i}(x)$.
For all $i\in I^{*}$, define a mapping $m_{i}^{N}:\mathbb{R}^{n}\rightarrow\mathbb{R}$
where $m_{i}^{N}(x)\triangleq\psi'(Ng_{i}(x))-1$. Then, define $m_{I^{*}}^{N}(\hat{x})\triangleq\left(m_{i}^{N}(\hat{x})\right)_{i\in I^{*}}$,
$E_{I^{*}}^{N}(\hat{x})\triangleq\mathrm{diag}(m_{I^{*}}^{N}(\hat{x})$),
and
\[
a_{N}(\hat{x},\lambda)\triangleq\left[\begin{array}{c}
\nabla_{x}\mathcal{L}_N(\hat{x},\lambda)+h_{N}(\hat{x},\lambda_{[m]\setminus I^{*}})\\
-\psi''(0)^{-1}N^{-1}\left(\mathrm{I}_{\left|I^{*}\right|}+E_{I^{*}}(\hat{x},N)\right)\left(\lambda_{I^{*}}^{*}-\lambda_{I^{*}}\right)
\end{array}\right].
\]

The following lemma helps to bound $\left\Vert \Delta z_{I^{*}}\right\Vert _{\infty}$.
\begin{lem}
\label{lem:delta_z_I*} Suppose Assumptions \ref{solandSlater}, \ref{assu:functions assump},
and \ref{regularity} hold. For any $\lambda\in\mathbb{R}_{++}^{m}$
and sufficiently large $N$, we have
\begin{equation}
\Phi_{N}(x^{*},\lambda^{*})\Delta z_{I^{*}}=a_{N}(\hat{x},\lambda)+r(\Delta z_{I^{*}}),\label{eq:key equality}
\end{equation}
where $r:\mathbb{R}^{n+\left|I^{*}\right|}\rightarrow\mathbb{R}^{n+\left|I^{*}\right|}$
and there exists $R>0$ such that $\left\Vert r(\Delta z_{I^{*}})\right\Vert _{\infty}\leq R \left\Vert \Delta z_{I^{*}}\right\Vert _{\infty}^{2}/2$.
\end{lem}

\begin{proof}
Fix an arbitrary small $\eta>0$. For sufficiently small $\epsilon>0$ and sufficiently large $N$, we
have $\hat{x}\in B_{\eta}(x^{*})$ and $\hat{\lambda}=\lambda\psi'(NG(\hat{x}))\in B_{\eta}(\lambda^{*})$,
due to properties (i), (ii), (iv), and (v) of the nonlinear rescaling
function class $\Psi$.

Due to Assumption \ref{assu:functions assump}, we have by Taylor's
theorem
\begin{align}
\nabla f(\hat{x}) & =\nabla f(x^{*})+\nabla^{2}f(x^{*})\Delta x+r_{0}(\Delta x),\label{eq:taylor expansion for f}\\
\nabla g_{i}(\hat{x}) & =\nabla g_{i}(x^{*})+\nabla^{2}g_{i}(x^{*})\Delta x+r_{i}(\Delta x),\quad i\in[m],\label{eq:taylor expansion for g_i}
\end{align}
where $r_{0}(0)=0$, $r_{i}(0)=0$,
\begin{equation}
\left\Vert r_{0}(\Delta x)\right\Vert _{\infty}\leq L_{f}^{(2)}\left\Vert \Delta x\right\Vert _{\infty}^{2}/2,\quad\mbox{and}\quad\left\Vert r_{i}(\Delta x)\right\Vert _{\infty}\leq L_{g}^{(2)}\left\Vert \Delta x\right\Vert _{\infty}^{2}/2,\quad\forall i\in[m].\label{eq:r0ri}
\end{equation}
Then
\begin{align*}
\nabla_{x}\mathcal{L}_{N}(\hat{x},\lambda) & =\nabla f(\hat{x})-\sum_{i\in [m]}\lambda_{i}\psi'(Ng_{i}(\hat{x}))\nabla g_{i}(\hat{x})\\
 & =\nabla f(\hat{x})-\sum_{i\in [m]}\hat{\lambda}_{i}\nabla g_{i}(\hat{x})\\
 & =\nabla f(\hat{x})-\sum_{i\in I^{*}}\left(\lambda_{i}^{*}+\Delta\lambda_{i}\right)\nabla g_{i}(\hat{x})-h_{N}(\hat{x},\lambda_{[m]\setminus I^{*}}).
\end{align*}
Using Eqs.~\eqref{eq:taylor expansion for f}-\eqref{eq:taylor expansion for g_i}
and the KKT conditions at the optimal solution $x^{*}$, we have
\begin{align}
\nabla_{x}\mathcal{L}_{N}(\hat{x},\lambda) & =\nabla f(x^{*})+\nabla^{2}f(x^{*})\Delta x+r_{0}(\Delta x)-\sum_{i\in I^{*}}\left(\lambda_{i}^{*}+\Delta\lambda_{i}\right)\left(\nabla g_{i}(x^{*})+\nabla^{2}g_{i}(x^{*})\Delta x+r_{i}(\Delta x)\right)\nonumber \\
 & \quad-h_{N}(\hat{x},\lambda_{[m]\setminus I^{*}})\nonumber \\
 & =\nabla f(x^{*})-\left\langle \lambda_{I^{*}}^{*},\nabla G_{I^{*}}(x^{*})\right\rangle +\left(\nabla^{2}f(x^{*})-\sum_{i\in I^{*}}\lambda_{i}^{*}\nabla^{2}g_{i}(x^{*})\right)\Delta x-\nabla G_{I^{*}}(x^{*})^{T}\Delta\lambda_{I^{*}}\nonumber \\
 & \quad+r^{(1)}(\Delta z_{I^{*}})-h_{N}(\hat{x},\lambda_{[m]\setminus I^{*}})\nonumber \\
 & =\nabla_{xx}L(x^{*},\lambda^{*})\Delta x-\nabla G_{I^{*}}(x^{*})^{T}\Delta\lambda_{I^{*}}+r^{(1)}(\Delta z_{I^{*}})-h_{N}(\hat{x},\lambda_{[m]\setminus I^{*}}),\label{eq:key 1}
\end{align}
where $r^{(1)}(\Delta z_{I^{*}})\triangleq r_{0}(\Delta x)-\sum_{i\in I^{*}}\Delta\lambda_{i}\nabla^{2}g_{i}(x^{*})\Delta x-\sum_{i\in I^{*}}\left(\lambda_{i}^{*}+\Delta\lambda_{i}\right)r_{i}(\Delta x)$.
From Eq.~\eqref{eq:r0ri}, we have $r^{(1)}(0)=0$ and there exists
$L^{(1)}>0$ such that
\begin{equation}
\|r^{(1)}(\Delta z_{I^{*}})\|_{\infty}\leq L^{(1)}\left\Vert \Delta z_{I^{*}}\right\Vert _{\infty}^{2}/2.\label{eq:r(1)}
\end{equation}

Using the fact that $\hat{\lambda}_{i}-\lambda_{i}=\lambda_{i}m_{i}^{N}(\hat{x})$,
we have $\lambda_{i}^{*}m_{i}^{N}(\hat{x})-\Delta\lambda_{i}=\left(1+m_{i}^{N}(\hat{x})\right)\left(\lambda_{i}^{*}-\lambda_{i}\right)$ for all $i\in I^{*}$, i.e.,
\begin{equation}
\varLambda_{I^{*}}^{*}m_{I^{*}}^{N}(\hat{x})-\Delta\lambda_{I^{*}}=\left(\mathrm{I}_{\left|I^{*}\right|}+E_{I^{*}}^{N}(\hat{x})\right)\left(\lambda_{I^{*}}^{*}-\lambda_{I^{*}}\right).\label{eq:lambda differentiate}
\end{equation}
Moreover, $m_{i}^{N}(\hat{x})=m_{i}^{N}(x^{*})+\left\langle \nabla m_{i}^{N}(x^{*}),\Delta x\right\rangle +r_{i}^{m}(\Delta x)$ for all $i\in I^{*}$, where $r_{i}^{m}(0)=0$. Due to the fact $\hat{x}\in B_{\eta}(x^{*})$
for small enough $\eta>0$, and by property (v) of the nonlinear rescaling
function class $\Psi$, there exists $L_{i}^{m}>0$ such that
\begin{equation}
\left|r_{i}^{m}(\Delta x)\right|\leq L_{i}^{m}\left\Vert \Delta x\right\Vert _{\infty}^{2}/2,\quad i\in I^{*}.\label{eq:ri^e}
\end{equation}
Noting that $g_{i}(x^{*})=0$ for all $i\in I^{*}$ and using property
(ii) of $\Psi$, we have
\[
m_{i}^{N}(x^{*})=\psi'(Ng_{i}(x^{*}))-1=\psi'(0)-1=0,\quad i\in I^{*},
\]
\[
\nabla m_{i}^{N}(x^{*})=\psi''(Ng_{i}(x^{*}))N\nabla g_{i}(x^{*})=\psi''(0)N\nabla g_{i}(x^{*}),\quad i\in I^{*}.
\]
Therefore, we have
\begin{equation}
m_{i}^{N}(\hat{x})=\psi''(0)N\left\langle \nabla g_{i}(x^{*}),\Delta x\right\rangle +r_{i}^{m}(\Delta x),\quad i\in I^{*}.\label{eq:ei differentiate}
\end{equation}

From Eqs.~(\ref{eq:lambda differentiate})-(\ref{eq:ei differentiate}),
we obtain
\begin{align}
-\varLambda_{I^{*}}^{*}\nabla G_{I^{*}}(x^{*})\Delta x+\psi''(0)^{-1}N^{-1}\Delta\lambda_{I^{*}}=-\psi''(0)^{-1}N^{-1}\left(\mathrm{I}_{\left|I^{*}\right|}+E_{I^{*}}^{N}(\hat{x})\right)\left(\lambda_{I^{*}}^{*}-\lambda_{I^{*}}\right)+r^{(2)}(\Delta x),\label{eq:key 2}
\end{align}
where $r^{(2)}(\Delta x)\triangleq\psi''(0)^{-1}N^{-1}\varLambda_{I^{*}}^{*}r_{I^{*}}^{m}(\Delta x)$,
$r_{I^{*}}^{m}(\Delta x)\triangleq\left(r_{i}^{m}(\Delta x)\right)_{i\in I^{*}}$,
and $r_{I^{*}}^{m}(0)=0$. From Eq.~\eqref{eq:ri^e}, we note that
there exists $L^{(2)}>0$ such that
\begin{equation}
\|r^{(2)}(\Delta x)\|_{\infty}\leq L^{(2)}\left\Vert \Delta x\right\Vert _{\infty}^{2}/2.\label{eq:r(2)}
\end{equation}
Rearranging Eq.~(\ref{eq:key 1}) and Eq.~(\ref{eq:key 2}), we have $\Phi_{N}(x^{*},\lambda^{*})\Delta z_{I^{*}}=a_{N}(\hat{x},\lambda)+r(\Delta z_{I^{*}})$,
where
\[
r(\Delta z_{I^{*}})\triangleq\left[\begin{array}{c}
-r^{(1)}(\Delta z_{I^{*}})\\
r^{(2)}(\Delta x)
\end{array}\right].
\]
From Eqs.~(\ref{eq:r(1)}) and (\ref{eq:r(2)}), we have that there
exists $R>0$ such that $\left\Vert r(\Delta z_{I^{*}})\right\Vert _{\infty}\leq R\left\Vert \Delta z_{I^{*}}\right\Vert _{\infty}^{2}/2$.
\end{proof}

Next, we are going to give the proofs of Lemmas \ref{lem:one step NR} and \ref{lem:bddness under inexact NR}.
\begin{proof}[Proof of Lemma \ref{lem:one step NR}]
Fix an arbitrary small $\eta>0$. For sufficiently small $\epsilon>0$ and sufficiently large $N$, we
have $\hat{x}\in B_{\eta}(x^{*})$ and $\hat{\lambda}=\lambda\psi'(NG(\hat{x}))\in B_{\eta}(\lambda^{*})$,
due to properties (i), (ii), (iv), and (v) of the nonlinear rescaling
function class $\Psi$. Noting that $g_{i}(x^{*})\geq\sigma$ for
all $i\in[m]\setminus I^{*}$, $\hat{x}\in B_{\eta}(x^{*})$ for small
enough $\eta>0$ and property (v) of $\Psi$, we obtain that
\[
\widehat{\lambda}_{i}=\lambda_{i}\psi'(Ng_{i}(\widehat{x}))\leq2\lambda_{i}\psi'(Ng_{i}(x^{*}))\leq2\lambda_{i}\psi'(N\sigma)\leq2d_{1}\sigma^{-1}N^{-1}\lambda_{i},\quad i\in[m]\setminus I^{*}.
\]
Hence we have
\begin{equation}
\|\widehat{\lambda}_{[m]\setminus I^{*}}-\lambda_{[m]\setminus I^{*}}^{*}\|_{\infty}\leq2d_{1}\sigma^{-1}N^{-1}\|\lambda_{[m]\setminus I^{*}}-\lambda_{[m]\setminus I^{*}}^{*}\|_{\infty}.\label{eq:=00003D00003DClambda_m-r}
\end{equation}

In the following, we will estimate $\left\Vert a_{N}(\hat{x},\lambda)\right\Vert _{\infty}$
and then estimate $\left\Vert \Delta z_{I^{*}}\right\Vert _{\infty}$
based on Eq.~(\ref{eq:key equality}). First, from the fact that
$g_{i}(x^{*})\geq\sigma$ for all $i\in[m]\setminus I^{*}$, $\hat{x}\in B_{\eta}(x^{*})$
for small enough $\eta>0$, the Lipschitz continuity of $\nabla g_{i}(\cdot)$,
and properties (i) and (iii) of $\Psi$, we have
\begin{align*}
\left\Vert h_{N}(\hat{x},\lambda_{[m]\setminus I^{*}})\right\Vert _{\infty} & =\left\Vert \sum_{i\in[m]\setminus I^{*}}\lambda_{i}\psi'(Ng_{i}(\hat{x}))\nabla g_{i}(\hat{x})\right\Vert _{\infty}\\
 & \leq\sum_{i\in[m]\setminus I^{*}}4\lambda_{i}\psi'(N\sigma/2)\left\Vert \nabla g_{i}(x^{*})\right\Vert _{\infty}.
\end{align*}
From property (v) of $\Psi$, there is a constant $C_{\nabla}>0$
such that $\sum_{i\in[m]\setminus I^{*}}4\psi'(N\sigma/2)\left\Vert \nabla g_{i}(x^{*})\right\Vert _{\infty}\leq C_{\nabla}N^{-1}$.
Therefore,
\begin{equation}
\left\Vert h_{N}(\hat{x},\lambda_{[m]\setminus I^{*}})\right\Vert _{\infty}\leq C_{\nabla}N^{-1}\|\lambda_{[m]\setminus I^{*}}-\lambda_{[m]\setminus I^{*}}^{*}\|_{\infty}\leq C_{\nabla}N^{-1}\left\Vert \lambda-\lambda^{*}\right\Vert _{\infty}.\label{eq:norm of h_N}
\end{equation}
Moreover,
\[
\mathrm{I}_{\left|I^{*}\right|}+E_{I^{*}}^{N}(\hat{x})=\mathrm{diag}(\left(\psi'(Ng_{i}(\hat{x}))\right)_{i\in I^{*}}).
\]
Using the fact that $\hat{x}\in B_{\eta}(x^{*})$ for small enough
$\eta>0$, $g_{i}(x^{*})=0$ for all $i\in I^{*}$, the continuity
of $\psi'(\cdot)$, $\psi'(0)=1$ from property (ii), and $\psi''(0)<0$
from property (iii) of $\Psi$, we have
\begin{equation}
\left\Vert -\psi''(0)^{-1}N^{-1}\left(\mathrm{I}_{\left|I^{*}\right|}+E_{I^{*}}^{N}(\hat{x})\right)\left(\lambda_{I^{*}}^{*}-\lambda_{I^{*}}\right)\right\Vert \leq-2\psi''(0)^{-1}N^{-1}\left\Vert \lambda-\lambda^{*}\right\Vert_{\infty} .\label{eq:norm of a2}
\end{equation}

From Eqs.~(\ref{eq:inexact primal update}), (\ref{eq:norm of h_N}),
and (\ref{eq:norm of a2}), we obtain
\[
\left\Vert a_{N}(\hat{x},\lambda)\right\Vert _{\infty}\leq\epsilon+\left(C_{\nabla}-2\psi''(0)^{-1}\right)N^{-1}\left\Vert \lambda-\lambda^{*}\right\Vert _{\infty}.
\]
From Eqs.~(\ref{eq:inverseofPHIbounded}) and (\ref{eq:key equality}),
we have for sufficiently large $N\geq N_{0}$,
\begin{align*}
\left\Vert \Delta z_{I^{*}}\right\Vert _{\infty} & \leq\left\Vert \Phi_{N}(x^{*},\lambda^{*})^{-1}\right\Vert \left(\left\Vert a_{N}(\hat{x},\lambda)\right\Vert _{\infty}+\left\Vert r(\Delta z_{I^{*}})\right\Vert _{\infty}\right)\\
 & \leq C_{\Phi}\left(\epsilon+\left(C_{\nabla}-2\psi''(0)^{-1}\right)N^{-1}\left\Vert \lambda-\lambda^{*}\right\Vert _{\infty}+R\left\Vert \Delta z_{I^{*}}\right\Vert _{\infty}^{2}/2\right).
\end{align*}
After rearrangement of the quadratic in $\left\Vert \Delta z_{I^{*}}\right\Vert _{\infty}$,
we obtain
\[
\left\Vert \Delta z_{I^{*}}\right\Vert _{\infty}\leq\left(1-\sqrt{1-2C_{\Phi}^{2}R(\epsilon+(C_{\nabla}-2\psi''(0)^{-1})N^{-1}\left\Vert \lambda-\lambda^{*}\right\Vert _{\infty})}\right)/(C_{\Phi}R).
\]
For sufficiently large $N$, we have
\[
1-2C_{\Phi}^{2}R\left(\epsilon+(C_{\nabla}-2\psi''(0)^{-1})N^{-1}\left\Vert \lambda-\lambda^{*}\right\Vert _{\infty}\right)\leq\sqrt{1-2C_{\Phi}^{2}R\left(\epsilon+(C_{\nabla}-2\psi''(0)^{-1})N^{-1}\left\Vert \lambda-\lambda^{*}\right\Vert _{\infty}\right)},
\]
thus
\begin{equation}
\left\Vert \Delta z_{I^{*}}\right\Vert _{\infty}\leq2C_{\Phi}\epsilon+2(C_{\nabla}-2\psi''(0)^{-1})C_{\Phi}N^{-1}\left\Vert \lambda-\lambda^{*}\right\Vert _{\infty}.\label{eq:z_r}
\end{equation}

From inequalities (\ref{eq:=00003D00003DClambda_m-r}) and (\ref{eq:z_r}),
and $c_{R}=\max\left\{ 2d_{1}\sigma^{-1},2(C_{\nabla}-2\psi''(0)^{-1})C_{\Phi}\right\} $,
we have
\begin{align}
\max\left\{ \left\Vert \hat{x}-x^{*}\right\Vert _{\infty},\|\hat{\lambda}-\lambda^{*}\|_{\infty}\right\} \leq2C_{\Phi}\epsilon+c_{R}N^{-1}\left\Vert \lambda-\lambda^{*}\right\Vert _{\infty}.\label{eq:linear or superlinear rate}
\end{align}
\end{proof}

\begin{proof}[Proof of Lemma \ref{lem:bddness under inexact NR}]
For all $k\geq 1$, we iteratively apply Eq.~(\ref{eq:linear or superlinear rate})
to obtain 
\begin{align*}
\left\Vert x^{k}-x^{*}\right\Vert _{\infty} & \leq\left(1-(c_{R}N^{-1})^{k}\right)2C_{\Phi}\epsilon/(1-c_{R}N^{-1})+(c_{R}N^{-1})^{k}\left\Vert \lambda^{0}-\lambda^{*}\right\Vert _{\infty}\\
 & \leq2C_{\Phi}\epsilon/(1-c_{R}N^{-1})+c_{R}N^{-1}\left\Vert \lambda^{0}-\lambda^{*}\right\Vert _{\infty}.
\end{align*}
\end{proof}

\subsection{Proof of Lemma \ref{lem:complexity_SGD}}
We divide this proof into two parts.

(i) At iteration $k=0$, we obtain $x^{1}$ from running $\mathcal{A}$
with $J_{0}^{SL}(K,\epsilon,\delta)$ iterations. We show that $x^{1}$
satisfies Eq.~(\ref{eq:inexact primal update}) with probability
at least $(1-\delta)^{1/K}$. First,
\begin{align}
\mathrm{P}\left(\left\Vert \nabla_{x}\mathcal{L}_{N}(x^{1},\lambda^{0})\right\Vert _{\infty}>\epsilon\right) & \overset{(a)}{\leq}\mathrm{P}\left(\left\Vert \nabla_{x}\mathcal{L}_{N}(x^{1},\lambda^{0})\right\Vert _{2}>\epsilon\right)\nonumber \\
 & \overset{(b)}{\leq}\mathrm{P}\left(\left\Vert x^{1}-x_{*}^{1}(\lambda^{0})\right\Vert _{2}>\epsilon/L_{N}(\lambda^{0})\right)\nonumber \\
 & \overset{(c)}{\leq}\mathbb{E}\left[\left\Vert x^{1}-x_{*}^{1}(\lambda^{0})\right\Vert _{2}^{2}\right]/(\epsilon/L_{N}(\lambda^{0}))^{2}\nonumber \\
 & \overset{(d)}{\leq}\frac{A\left\Vert x^{0}-x_{*}^{1}(\lambda^{0})\right\Vert _{2}^{2}+B}{J_{0}^{SL}(K,\epsilon,\delta)(\epsilon/L_{N}(\lambda^{0}))^{2}}\nonumber \\
 & \overset{(e)}{\leq}\frac{A\left(\left\Vert x^{0}-x^{*}\right\Vert _{2}^{2}+\left\Vert x_{*}^{1}(\lambda^{0})-x^{*}\right\Vert _{2}^{2}\right)+B}{J_{0}^{SL}(K,\epsilon,\delta)(\epsilon/L_{N}(\lambda^{0}))^{2}}\nonumber \\
 & \overset{(f)}{\leq}\frac{nA\left(\left\Vert x^{0}-x^{*}\right\Vert _{\infty}^{2}+\left\Vert x_{*}^{1}(\lambda^{0})-x^{*}\right\Vert _{\infty}^{2}\right)+B}{J_{0}^{SL}(K,\epsilon,\delta)(\epsilon/L_{N}(\lambda^{0}))^{2}},\label{eq:mass inequ 1}
\end{align}
where $(a)$ follows from $\left\Vert \cdot\right\Vert _{\infty}\leq\left\Vert \cdot\right\Vert _{2}$,
$(b)$ holds due to $\nabla_{x}\mathcal{L}_{N}(\cdot,\lambda^{0})$
is Lipschitz continuous with parameter $L_{N}(\lambda^{0})$, $(c)$
follows from Markov's inequality, $(d)$ is due to Assumption \ref{assu:rate}(i),
$(e)$ follows from triangle inequality, $(f)$ is due to $\left\Vert x\right\Vert _{2}^{2}\leq n\left\Vert x\right\Vert _{\infty}^{2}$
for any $x\in\mathbb{R}^{n}$.

Now, using $\left\Vert x_{*}^{1}(\lambda^{0})-x^{*}\right\Vert _{\infty}\leq c_{R}N^{-1}\left\Vert \lambda^{0}-\lambda^{*}\right\Vert _{\infty}$
from Theorem \ref{lem:one step NR}, and the definition of $J_{0}^{SL}(K,\epsilon,\delta)$
in Eq.~(\ref{eq:J0SGD}), we have $\mathrm{P}\left(\left\Vert \nabla_{x}\mathcal{L}_{N}(x^{1},\lambda^{0})\right\Vert _{\infty}>\epsilon\right)\leq1-(1-\delta)^{1/K}$.

(ii) At iteration $k_{0}\geq1$, suppose the solution\textbf{ }returned
from the subroutine $\mathcal{A}$ satisfies Eq.~(\ref{eq:inexact primal update})
in all previous iterations, we run $\mathcal{A}$ with $J_{k_{0}}^{SL}(K,\epsilon,\delta)$
iterations at the current iteration. In the following, we show that\textbf{
}$x^{k_{0}+1}$ satisfies Eq.~(\ref{eq:inexact primal update}) with
probability at least $(1-\delta)^{1/K}$. First, using the same argument
as in inequality (\ref{eq:mass inequ 1}), we have
\begin{align*}
\mathrm{P}\left(\left\Vert \nabla_{x}\mathcal{L}_{N}(x^{k_{0}+1},\lambda^{k_{0}})\right\Vert _{\infty}>\epsilon\right) & \leq\frac{nA\left(\mathbb{E}\left[\left\Vert x^{k_{0}}-x^{*}\right\Vert _{\infty}^{2}\right]+\mathbb{E}\left[\left\Vert x_{*}^{k_{0}+1}(\lambda^{k_{0}})-x^{*}\right\Vert _{\infty}^{2}\right]\right)+B}{J_{k_{0}}^{SL}(K,\epsilon,\delta)(\epsilon/L_{N}(\lambda^{k_{0}}))^{2}}.
\end{align*}
Now, due to the assumption that the solution\textbf{ }returned from
the subroutine $\mathcal{A}$ satisfies Eq.~(\ref{eq:inexact primal update})
in all previous iterations, we iteratively apply Theorem \ref{lem:one step NR}
to obtain
\begin{align*}
\left\Vert x^{k_{0}}-x^{*}\right\Vert _{\infty} & \leq\left(1-(c_{R}N^{-1})^{k_{0}}\right)2C_{\Phi}\epsilon/(1-c_{R}N^{-1})+(c_{R}N^{-1})^{k_{0}}\left\Vert \lambda^{0}-\lambda^{*}\right\Vert _{\infty}\\
 & \leq2C_{\Phi}\epsilon/(1-c_{R}N^{-1})+(c_{R}N^{-1})^{k_{0}}\left\Vert \lambda^{0}-\lambda^{*}\right\Vert _{\infty},
\end{align*}
and
\[
\left\Vert x_{*}^{k_{0}+1}(\lambda^{k_{0}})-x^{*}\right\Vert _{\infty}\leq2C_{\Phi}c_{R}N^{-1}\epsilon/(1-c_{R}N^{-1})+(c_{R}N^{-1})^{k_{0}+1}\left\Vert \lambda^{0}-\lambda^{*}\right\Vert _{\infty}.
\]
Moreover, using the definition of $J_{k}(K,\epsilon,\delta)$ in
Eq.~(\ref{eq:JkSGD}), we have $\mathrm{P}\left(\left\Vert \nabla_{x}\mathcal{L}_{N}(x^{k_{0}+1},\lambda^{k_{0}})\right\Vert _{\infty}>\epsilon\right)\leq1-(1-\delta)^{1/K}$.

\section{Randomized First-order Algorithms with Generic Sampling}
We recall that Problem
$\mathbb{P}_{N}(\lambda)$ can be rewritten as:
\begin{equation}
\min_{x\in\mathcal{X}}\left\{ \mathcal{L}_{N}(x,\lambda)\equiv\sum_{i\in[m]}\left(\lambda_{i}/\|\lambda\|_{1}\right)f_{i}^{N}(x;\lambda)\right\} .\label{eq:constrained problem}
\end{equation}
As stated in Subsection \ref{subsec:Technical-assumptions}, for any
fixed $\lambda\in\mathbb{R}_{++}^{m}$, the primal update $\min_{x\in\mathcal{X}}\mathcal{L}_{N}(x,\lambda)$
has a unique solution because $f(\cdot)$ is strongly convex.

We can rewrite the optimality condition of Problem (\ref{eq:constrained problem})
in a more convenient form using the monotone operators $B_{i}(x,\lambda)\triangleq\nabla_{x}f_{i}^{N}(x;\lambda)=\nabla f(x)-\|\lambda\|_{1}\psi'(Ng_{i}(x))\nabla g_{i}(x)$
for all $i\in[m]$. Then, the unique solution of Problem (\ref{eq:constrained problem})
satisfies the monotone inclusion:
\begin{equation}
0\in\partial I_{\mathcal{X}}(x)+B(x,\lambda),\label{eq:zeroofO}
\end{equation}
where $B(x,\lambda)\triangleq\sum_{i\in[m]}\wp_{i}(\lambda)B_{i}(x,\lambda)$.
We have $\left\langle B(x,\lambda)-B(y,\lambda),x-y\right\rangle \geq\mu_{f}\left\Vert x-y\right\Vert _{2}^{2}$
for all $x,\,y\in\mathbb{R}^{n}$ by strong convexity of $f$. We
recall that by Assumption \ref{assu:conceptual}, $L_{N}(\lambda)>0$
is an upper bound on the Lipschitz constant of $B_{i}(\cdot,\lambda)$
for all $i\in[m]$ over $\mathcal{X}$.

We adopt a generic framework for $\mathcal{A}^{\mathcal{Q}}$ which
encompasses unbiased first-order methods like SGD, SVRG, SAGA, etc.
Let $\{y_{t}\}_{t\geq0}$ denote the iterates of $\mathcal{A}^{\mathcal{Q}}$,
and let $\varphi^{t}=\left\{ \varphi_{i}^{t}\right\} _{i\in[m]}$
be a collection of auxiliary variables (which serve as proxies for
past gradient evaluations).

Suppose that $\left\{ I_{t}\right\} _{t\geq0}$ are i.i.d.\ following
a probability distribution $\mathcal{Q}=(q_{i})_{i\in[m]} \in \mathcal P_{+}([m])$. We define additional operators $A_{i}(x,\lambda)\triangleq\wp_{i}(\lambda)B_{i}(x,\lambda)/q_{i}$
for all $i\in[m]$. We will later need the fact that each $A_{i}(\cdot,\lambda)$ is $\wp_{i}(\lambda)L_{N}(\lambda)/q_{i}$-Lipschitz
continuous for all $i\in[m]$. Moreover, we define a gradient estimator
\begin{equation}
\mathcal{G}(y_{t},\,\varphi^{t},\,I_{t})\triangleq A_{I_{t}}(y_{t},\lambda)-\varphi_{I_{t}}^{t}+\sum_{i\in[m]}q_{i}\varphi_{i}^{t},\quad\forall\,t\geq0,\label{eq:random estimator}
\end{equation}
which is an unbiased estimator of $B(y_{t},\lambda)$ (see the formal statement in Lemma~\ref{lem:unbiasednessofrandomestimator}). 

The primal sequence $\{y_{t}\}_{t\geq0}$ of $\mathcal{A}^{\mathcal{Q}}$
is updated according to:
\begin{equation}
y_{t+1}=\Pi_{\mathcal{X}}[y_{t}-\gamma_{t}\mathcal{G}(y_{t},\,\varphi^{t},\,I_{t})],\,\forall\,t\geq0,\label{eq:iterate updates in A}
\end{equation}
for step-sizes $\{\gamma_t\}_{t \geq 0}$ and the auxiliary variables are updated according to some generic
scheme:
\begin{equation}
\varphi^{t+1}=\mathcal{U}_{t}\left(y_{t},\,\varphi^{t},\,I_{t}\right),\,\forall\,t\geq0.\label{eq:updaterule_U}
\end{equation}
Let $\mathcal{F}_{t}\triangleq\sigma\left(y_{0},\varphi^{0},I_{0},\ldots,y_{t-1},\varphi^{t-1},I_{t-1},y_{t},\varphi^{t}\right)$
denote the history of $\mathcal{A}^{\mathcal{Q}}$ up to iteration
$t$, which forms a filtration.

\subsection{Basic Properties}

We first confirm that this construction of $\mathcal{G}(y_{t},\,\varphi^{t},\,I_{t})$
is an unbiased estimator.
\begin{lem}
\label{lem:unbiasednessofrandomestimator} Fix $\lambda\in\mathbb{R}_{++}^{m}$ and $\mathcal{Q} \in \mathcal P_{+}([m])$, then $\mathbb{E}^{\mathcal{Q}}\left[\mathcal{G}(y_{t},\,\varphi^{t},\,I_{t})\mid\mathcal{F}_{t}\right]=B(y_{t},\lambda)$ for all $t \geq 0$.
\end{lem}

The expected distance to the solution $x^{*}(\lambda)$ of Eq.~(\ref{eq:zeroofO})
contracts after each iteration, depending on the conditional variance
of $\mathcal{G}(y_{t},\,\varphi^{t},\,I_{t})$.
\begin{lem}
\label{lem:convergence depends on variance} Fix $\lambda\in\mathbb{R}_{++}^{m}$ and $\mathcal{Q} \in \mathcal P_{+}([m])$,
and let $\left\{ y_{t}\right\}_{t\geq0}$ be
produced by Eq.~(\ref{eq:iterate updates in A}) using $\mathcal{A}^{\mathcal{Q}}$. Then, for all $t\geq0$,
\[
\mathbb{E}^{\mathcal{Q}}\left[\left\Vert y_{t+1}-x^{*}(\lambda)\right\Vert _{2}^{2}\mid\mathcal{F}_{t}\right]\leq\left(1-2\gamma_{t}\mu_{f}+\gamma_{t}^{2}L_{N}(\lambda)^{2}\right)\left\Vert y_{t}-x^{*}(\lambda)\right\Vert _{2}^{2}+\gamma_{t}^{2}\mathbb{E}^{\mathcal{Q}}\left[\left\Vert \mathcal{G}(y_{t},\,\varphi^{t},\,I_{t})-B(y_{t},\lambda)\right\Vert _{2}^{2}\mid\mathcal{F}_{t}\right].
\]
\end{lem}

\begin{proof}
Recall that $x^{*}(\lambda)$ is the unique solution of Eq.~(\ref{eq:zeroofO}),
i.e., $x^{*}(\lambda)=\Pi_{\mathcal X}\left[x^{*}(\lambda)-\gamma B(x^{*}(\lambda),\lambda)\right]$.
It follows that:
\begin{align*}
 & \left\Vert y_{t+1}-x^{*}(\lambda)\right\Vert _{2}^{2}\\
= & \left\Vert \Pi_{\mathcal{X}}[y_{t}-\gamma_{t}\mathcal{G}(y_{t},\,\varphi^{t},\,I_{t})]-\Pi_{\mathcal{X}}[x^{*}(\lambda)-\gamma B(x^{*}(\lambda),\lambda)]\right\Vert _{2}^{2}\\
\leq & \left\Vert (y_{t}-\gamma_{t}\mathcal{G}(y_{t},\,\varphi^{t},\,I_{t}))-(x^{*}(\lambda)-\gamma_{t}B(x^{*}(\lambda),\lambda))\right\Vert _{2}^{2}\\
= & \left\Vert (y_{t}-x^{*}(\lambda))-\gamma_{t}(B(y_{t},\lambda)-B(x^{*}(\lambda),\lambda))-\gamma_{t}(\mathcal{G}(y_{t},\,\varphi^{t},\,I_{t})-B(y_{t},\lambda))\right\Vert _{2}^{2}\\
= & \left\Vert y_{t}-x^{*}(\lambda)\right\Vert _{2}^{2}+\gamma_{t}^{2}\left\Vert B(y_{t},\lambda)-B(x^{*}(\lambda),\lambda)\right\Vert _{2}^{2}+\gamma_{t}^{2}\left\Vert \mathcal{G}(y_{t},\,\varphi^{t},\,I_{t})-B(y_{t},\lambda)\right\Vert _{2}^{2}\\
&-2\gamma_{t}\left\langle B(y_{t},\lambda)-B(x^{*}(\lambda),\lambda),y_{t}-x^{*}(\lambda)\right\rangle-2\gamma_{t}\left\langle \mathcal{G}(y_{t},\,\varphi^{t},\,I_{t})-B(y_{t},\lambda),y_{t}-x^{*}(\lambda)\right\rangle \\
 &+2\gamma_{t}^{2}\left\langle B(y_{t},\lambda)-B(x^{*}(\lambda),\lambda),\mathcal{G}(y_{t},\,\varphi^{t},\,I_{t})-B(y_{t},\lambda)\right\rangle ,
\end{align*}
where the inequality follows from non-expansiveness of the projection
operator. Since $\mathcal{G}(y_{t},\,\varphi^{t},\,I_{t})$
is a (conditionally) unbiased estimator of $B(y_{t},\lambda)$ by Lemma~\ref{lem:unbiasednessofrandomestimator},
we have
\begin{align*}
&\mathbb{E}^{\mathcal{Q}}\left[\left\Vert y_{t+1}-x^{*}(\lambda)\right\Vert _{2}^{2}\mid\mathcal{F}_{t}\right]\\
\leq & \left\Vert y_{t}-x^{*}(\lambda)\right\Vert _{2}^{2}+\gamma_{t}^{2}\left\Vert B(y_{t},\lambda)-B(x^{*}(\lambda),\lambda)\right\Vert _{2}^{2}-2\gamma_{t}\left\langle B(y_{t},\lambda)-B(x^{*}(\lambda),\lambda),y_{t}-x^{*}(\lambda)\right\rangle \\
 & +\gamma_{t}^{2}\mathbb{E}^{\mathcal{Q}}\left[\left\Vert \mathcal{G}(y_{t},\,\varphi^{t},\,I_{t})-B(y_{t},\lambda)\right\Vert _{2}^{2}\mid\mathcal{F}_{t}\right]\\
\leq & \left(1-2\gamma_{t}\mu_{f}+\gamma_{t}^{2}L_{N}(\lambda)^{2}\right)\left\Vert y_{t}-x^{*}(\lambda)\right\Vert _{2}^{2}+\gamma_{t}^{2}\mathbb{E}^{\mathcal{Q}}\left[\left\Vert \mathcal{G}(y_{t},\,\varphi^{t},\,I_{t})-B(y_{t},\lambda)\right\Vert _{2}^{2}\mid\mathcal{F}_{t}\right],
\end{align*}
where the second inequality is due to strong monotonicity of $B(\cdot,\lambda)$
and $L_{N}(\lambda)$-Lipschitz continuity of $B(\cdot,\lambda)$.
\end{proof}
We can upper bound the conditional variance of $\mathcal{G}(y_{t},\,\varphi^{t},\,I_{t})$.
\begin{lem}
\label{conditional variance_bound } Fix $\lambda\in\mathbb{R}_{++}^{m}$ and $\mathcal{Q} \in \mathcal P_{+}([m])$,
and let $\left\{ y_{t}\right\}_{t\geq0}$ be
produced by Eq.~(\ref{eq:iterate updates in A}) using $\mathcal{A}^{\mathcal{Q}}$. Then, for all $t\geq0$,
\[
\mathbb{E}^{\mathcal{Q}}\left[\left\Vert \mathcal{G}(y_{t},\,\varphi^{t},\,I_{t})-B(y_{t},\lambda)\right\Vert _{2}^{2}\mid\mathcal{F}_{t}\right]\leq2\left(r^{\mathcal{Q}}L_{N}(\lambda)^{2}\left\Vert y_{t}-x^{*}(\lambda)\right\Vert _{2}^{2}+\sum_{i\in[m]}q_{i}\left\Vert \varphi_{i}^{t}-B_{i}(x^{*}(\lambda),\lambda)\right\Vert _{2}^{2}\right).
\]
\end{lem}

\begin{proof}
Using $\mathbb{E}^{\mathcal{Q}}\left[A_{I_{t}}(y_{t},\lambda)\mid\mathcal{F}_{t}\right]=\sum_{i\in[m]}\wp_{i}(\lambda)B_{i}(y_{t},\lambda)=B(y_{t},\lambda)$
and $\mathbb{E}^{\mathcal{Q}}\left[\varphi_{I_{t}}^{t}\mid\mathcal{F}_{t}\right]=\sum_{i\in[m]}q_{i}\varphi_{i}^{t}$,
we have
\begin{align*}
\mathbb{E}^{\mathcal{Q}}\left[\left\Vert \mathcal{G}(y_{t},\,\varphi^{t},\,I_{t})-B(y_{t},\lambda)\right\Vert _{2}^{2}\mid\mathcal{F}_{t}\right] & =\mathbb{E}^{\mathcal{Q}}\left[\left\Vert A_{I_{t}}(y_{t},\lambda)-\varphi_{I_{t}}^{t}-\mathbb{E}^{\mathcal{Q}}\left[A_{I_{t}}(y_{t},\lambda)-\varphi_{I_{t}}^{t}\mid\mathcal{F}_{t}\right]\right\Vert _{2}^{2}\mid\mathcal{F}_{t}\right]\\
 & \leq\mathbb{E}^{\mathcal{Q}}\left[\left\Vert A_{I_{t}}(y_{t},\lambda)-\varphi_{I_{t}}^{t}\right\Vert _{2}^{2}\mid\mathcal{F}_{t}\right],
\end{align*}
because the conditional variance is bounded by the conditional second
moment. We further bound the term $\mathbb{E}^{\mathcal{Q}}\left[\left\Vert A_{I_{t}}(y_{t},\lambda)-\varphi_{I_{t}}^{t}\right\Vert _{2}^{2}\mid\mathcal{F}_{t}\right]$
as follows:
\begin{align*}
\mathbb{E}^{\mathcal{Q}}\left[\left\Vert A_{I_{t}}(y_{t},\lambda)-\varphi_{I_{t}}^{t}\right\Vert _{2}^{2}\mid\mathcal{F}_{t}\right] & =\mathbb{E}^{\mathcal{Q}}\left[\left\Vert \left(A_{I_{t}}(y_{t},\lambda)-A_{I_{t}}(x^{*}(\lambda),\lambda)\right)-\left(\varphi_{I_{t}}^{t}-A_{I_{t}}(x^{*}(\lambda),\lambda)\right)\right\Vert _{2}^{2}\mid\mathcal{F}_{t}\right]\\
 & \leq2\,\mathbb{E}^{\mathcal{Q}}\left[\left\Vert A_{I_{t}}(y_{t},\lambda)-A_{I_{t}}(x^{*}(\lambda),\lambda)\right\Vert _{2}^{2}+\left\Vert \varphi_{I_{t}}^{t}-A_{I_{t}}(x^{*}(\lambda),\lambda)\right\Vert _{2}^{2}\mid\mathcal{F}_{t}\right]\\
 & \leq2\left(r^{\mathcal{Q}}L_{N}(\lambda)^{2}\left\Vert y_{t}-x^{*}(\lambda)\right\Vert _{2}^{2}+\sum_{i\in[m]}q_{i}\left\Vert \varphi_{i}^{t}-A_{i}(x^{*}(\lambda),\lambda)\right\Vert _{2}^{2}\right),
\end{align*}
where the first inequality follows from the fact that $(a-b)^{2}\leq2(a^{2}+b^{2})$,
and the second inequality is due to the $\wp_{i}(\lambda)L_{N}(\lambda)/q_{i}$-Lipschitz
continuity of each $A_{i}(\cdot,\lambda)$ for all $i\in[m]$.
\end{proof}
Combining Lemmas \ref{lem:convergence depends on variance} and \ref{conditional variance_bound }
gives the following result.
\begin{cor}
\label{cor:rand_1st_alg_onestep} Fix $\lambda\in\mathbb{R}_{++}^{m}$ and $\mathcal{Q} \in \mathcal P_{+}([m])$,
and let $\left\{ y_{t}\right\}_{t\geq0}$ be
produced by Eq.~(\ref{eq:iterate updates in A}) using $\mathcal{A}^{\mathcal{Q}}$. Then, for all $t\geq0$,
\begin{align*}
&\mathbb{E}^{\mathcal{Q}}\left[\left\Vert y_{t+1}-x^{*}(\lambda)\right\Vert _{2}^{2}\right]\\
\leq&\left(1-2\gamma_{t}\mu_{f}+\left(1+2r^{\mathcal{Q}}\right)\gamma_{t}^{2}L_{N}(\lambda)^{2}\right)\mathbb{E}^{\mathcal{Q}}\left[\left\Vert y_{t}-x^{*}(\lambda)\right\Vert _{2}^{2}\right]+2\,\gamma_{t}^{2}\mathbb{E}^{\mathcal{Q}}\left[\sum_{i\in[m]}q_{i}\left\Vert \varphi_{i}^{t}-A_{i}(x^{*}(\lambda),\lambda)\right\Vert _{2}^{2}\right].
\end{align*}
\end{cor}

\begin{rem}
\label{rem:special cases of sampling} (i) If we take $\mathcal{Q}=\wp(\lambda)$,
then $q_{i}=\wp_{i}(\lambda)$ for all $i\in[m]$ and $r^{\mathcal{Q}}=\sum_{i\in[m]}\wp_{i}(\lambda)=1$.

(ii) If $\mathcal{Q}$ is uniform, then $q_{i}=1/m$ for all $i\in[m]$
and $r^{\mathcal{Q}}=\sum_{i\in[m]}m\,\wp_{i}(\lambda)^{2}=m\,\|\lambda\|_{2}^{2}/\|\lambda\|_{1}^{2}$.
\end{rem}

\subsection{SGD with Generic Sampling}

SGD has a sublinear convergence rate in expectation.
\begin{thm}
\label{thm:convgofSGD} Let $M_{B}\triangleq\max_{i\in[m]}\left\Vert B_{i}(x^{*}(\lambda),\lambda)\right\Vert _{2}$, and suppose $\mathcal{A}_{SGD}^{\mathcal Q}$ is run with decreasing stepsizes $\gamma_{t}=\frac{2}{\mu_{f}(t+2(1+2r^{\mathcal{Q}})L_{N}(\lambda)^{2}/\mu_{f}^{2})}$ for all $t\geq0$. Then,
\[
\mathbb{E}^{\mathcal{Q}}\left[\left\Vert y_{t}-x^{*}(\lambda)\right\Vert _{2}^{2}\right]\leq\frac{(2(1+2r^{\mathcal{Q}})L_{N}(\lambda)^{2}/\mu_{f}^{2}-1)\left\Vert y_{0}-x^{*}(\lambda)\right\Vert _{2}^{2}+8r^{\mathcal{Q}}M_{B}^{2}/\mu_{f}^{2}}{t},\quad\forall\,t\geq1.
\]
\end{thm}

\begin{proof}
Using Corollary \ref{cor:rand_1st_alg_onestep}, for all $t\geq0$ we have
\[
\mathbb{E}^{\mathcal{Q}}\left[\left\Vert y_{t+1}-x^{*}(\lambda)\right\Vert _{2}^{2}\right]\leq\left(1-2\gamma_{t}\mu_{f}+\left(1+2r^{\mathcal{Q}}\right)\gamma_{t}^{2}L_{N}(\lambda)^{2}\right)\mathbb{E}^{\mathcal{Q}}\left[\left\Vert y_{t}-x^{*}(\lambda)\right\Vert _{2}^{2}\right]+2\gamma_{t}^{2}r^{\mathcal{Q}}M_{B}^{2}.
\]
With the above choice of stepsizes $\{\gamma_{t}\}_{t\geq0}$, we
have
\[
1-2\gamma_{t}\mu_{f}+\left(1+2r^{\mathcal{Q}}\right)\gamma_{t}^{2}L_{N}(\lambda)^{2}=1-\gamma_{t}\mu_{f}-\left(1+2r^{\mathcal{Q}}\right)\gamma_{t}L_{N}(\lambda)^{2}\left(\frac{\mu_{f}}{\left(1+2\,r^{\mathcal{Q}}\right)L_{N}(\lambda)^{2}}-\gamma_{t}\right)\leq1-\gamma_{t}\mu_{f}.
\]
Therefore, for the constant $c\triangleq2(1+2r^{\mathcal{Q}})L_{N}(\lambda)^{2}/\mu_{f}^{2}$, we can write
\begin{align*}
\mathbb{E}^{\mathcal{Q}}\left[\left\Vert y_{t+1}-x^{*}(\lambda)\right\Vert _{2}^{2}\right] & \leq\left(1-\gamma_{t}\mu_{f}\right)\mathbb{E}^{\mathcal{Q}}\left[\left\Vert y_{t}-x^{*}(\lambda)\right\Vert _{2}^{2}\right]+2\gamma_{t}^{2}r^{\mathcal{Q}}M_{B}^{2}\\
 & =\frac{t-2+c}{t+c}\mathbb{E}^{\mathcal{Q}}\left[\left\Vert y_{t}-x^{*}(\lambda)\right\Vert _{2}^{2}\right]+\frac{8r^{\mathcal{Q}}}{\mu_{f}^{2}(t+c)^{2}}M_{B}^{2}.
\end{align*}
Applying the above inequality recursively gives
\begin{align*}
\mathbb{E}^{\mathcal{Q}}\left[\left\Vert y_{t}-x^{*}(\lambda)\right\Vert _{2}^{2}\right] & \leq\frac{(c-1)(c-2)}{(t-1+c)(t-2+c)}\left\Vert y_{0}-x^{*}(\lambda)\right\Vert _{2}^{2}+\frac{8r^{\mathcal{Q}}M_{B}^{2}}{\mu_{f}^{2}(t-1+c)(t-2+c)}\sum_{l=0}^{t-1}\frac{l+c-1}{l+c}\\
 & \leq\left((c-1)\left\Vert y_{0}-x^{*}(\lambda)\right\Vert _{2}^{2}+8r^{\mathcal{Q}}M_{B}^{2}/\mu_{f}^{2}\right)/t,
\end{align*}
where the last inequality follows from the fact that $\frac{l+c-1}{l+c}\leq1$
for all $l\in[t]_{0}$, $t(c-2)\leq(t-1+c)(t-2+c)$ and $t^{2}\leq(t-1+c)(t-2+c)$.
\end{proof}

\subsection{SVRG with Generic Sampling}

SVRG has a linear convergence rate in expectation. For a constant
step-size $\gamma_{t}=\gamma$ for all $t\geq0$, we let $\alpha_{SVRG}^{\mathcal{Q}}(\gamma)\triangleq1-2\gamma\mu_{f}+(1+2r^{\mathcal{Q}}+2Mr^{\mathcal{Q}})\gamma^{2}L_{N}(\lambda)^{2}$
be the contraction factor for the iterates of SVRG.
\begin{thm}
\label{thm:linear convergence rate of SVRG} Suppose $\mathcal{A}_{SVRG}^{\mathcal Q}$ is run with a constant stepsize $\gamma<\frac{2\mu_{f}}{(1+2r^{\mathcal{Q}}+2Mr^{\mathcal{Q}})L_{N}(\lambda)^{2}}$. Then $\alpha_{SVRG}^{\mathcal{Q}}(\gamma)\in(0,1)$ and
\[
\mathbb{\mathbb{E}^{\mathcal{Q}}}\left[\left\Vert y_{j}-x^{*}(\lambda)\right\Vert _{2}^{2}\right]\leq\alpha_{SVRG}^{\mathcal{Q}}(\gamma)^{j}\mathbb{E}^{\mathcal{Q}}\left[\left\Vert y_{0}-x^{*}(\lambda)\right\Vert _{2}^{2}\right],\quad\forall\,j\geq1.
\]
\end{thm}

\begin{proof}
Recall that we denote the inner iterates as $\{\tilde{y}_{l}\}_{l\geq0}$,
where $M\,j\leq l<M(j+1)$ correspond to epoch $j\geq1$. In SVRG,
the proxies $\{\varphi_{i}^{l}\}_{i\in[m]}$ are only updated at the
beginning of each epoch. Since $\{\varphi_{i}^{l}\}_{i\in[m]}=\{A_{i}(\tilde{y}_{M\,j},\lambda)\}_{i\in[m]}$
for $M\,j\leq l<M(j+1)$, we have:
\[
\sum_{i\in[m]}q_{i}\left\Vert \varphi_{i}^{l}-A_{i}(x^{*}(\lambda),\lambda)\right\Vert _{2}^{2}=\sum_{i\in[m]}q_{i}\left\Vert A_{i}(\tilde{y}_{M\,j},\lambda)-A_{i}(x^{*}(\lambda),\lambda)\right\Vert _{2}^{2}\leq r^{\mathcal{Q}}L_{N}(\lambda)^{2}\|\tilde{y}_{M\,j}-x^{*}(\lambda)\|_{2}^{2},
\]
where the inequality is due to the $\wp_{i}(\lambda)L_{N}(\lambda)/q_{i}$-Lipschitz
continuity of each $A_{i}$ for all $i\in[m]$. Using Corollary \ref{cor:rand_1st_alg_onestep}
and defining $\overline{\alpha}\triangleq1-2\gamma\mu_{f}+(1+2r^{\mathcal{Q}})\gamma^{2}L_{N}(\lambda)^{2}$,
we immediately have that for all $M\,j\leq l<M(j+1)$,

\begin{equation}
\mathbb{E}^{\mathcal{Q}}\left[\left\Vert \tilde{y}_{l+1}-x^{*}(\lambda)\right\Vert _{2}^{2}\right]\leq\overline{\alpha}\mathbb{E}^{\mathcal{Q}}\left[\left\Vert \tilde{y}_{l}-x^{*}(\lambda)\right\Vert _{2}^{2}\right]+2\gamma^{2}r^{\mathcal{Q}}L_{N}(\lambda)^{2}\mathbb{E}^{\mathcal{Q}}\left[\|\tilde{y}_{M\,j}-x^{*}(\lambda)\|_{2}^{2}\right].\label{eq:onestep}
\end{equation}
If the stepsize satisfies $\gamma<\frac{2\mu_{f}}{(1+2r^{\mathcal{Q}}+2Mr^{\mathcal{Q}})L_{N}(\lambda)^{2}}$,
then both $\alpha_{SVRG}^{\mathcal{Q}}(\gamma),\,\overline{\alpha}\in(0,1)$.
By recursively applying inequality~(\ref{eq:onestep}), we obtain:
\begin{align}
\mathbb{E}^{\mathcal{Q}}\left[\left\Vert \tilde{y}_{M(j+1)}-x^{*}(\lambda)\right\Vert _{2}^{2}\right] & \leq\bar{\alpha}^{M}\mathbb{E}^{\mathcal{Q}}\left[\left\Vert \tilde{y}_{M\,j}-x^{*}(\lambda)\right\Vert _{2}^{2}\right]+2\gamma^{2}r^{\mathcal{Q}}L_{N}(\lambda)^{2}\mathbb{E}^{\mathcal{Q}}\left[\|\tilde{y}_{M\,j}-x^{*}(\lambda)\|_{2}^{2}\right]\sum_{l=0}^{M-1}\overline{\alpha}^{l}\nonumber \\
 & \leq\left(\bar{\alpha}+2M\gamma^{2}r^{\mathcal{Q}}L_{N}(\lambda)^{2}\right)\mathbb{E}^{\mathcal{Q}}\left[\left\Vert \tilde{y}_{M\,j}-x^{*}(\lambda)\right\Vert _{2}^{2}\right]\nonumber \\
 & =\alpha_{SVRG}^{\mathcal{Q}}(\gamma)\,\mathbb{E}^{\mathcal{Q}}\left[\left\Vert \tilde{y}_{M\,j}-x^{*}(\lambda)\right\Vert _{2}^{2}\right].\label{eq:geometricallyacrossepoch}
\end{align}
The desired result then follows by recursively applying inequality~(\ref{eq:geometricallyacrossepoch}).
\end{proof}

\end{document}